\documentclass[11pt]{article}

\usepackage{lineno,hyperref}
\modulolinenumbers[5]

\usepackage{epsfig,psfrag}
\usepackage{amsmath,amsthm}
\usepackage{amssymb}
\usepackage{graphicx}
\usepackage{subcaption}
\usepackage{latexsym,pifont,color,comment}
\usepackage{mathtools}
\usepackage{placeins}
\usepackage[utf8]{inputenc}
\usepackage[english]{babel}

\def\mbbR{\mathbb{R}}
\def\mbbN{\mathbb{N}}

\def\rarrow{\rightarrow}

\def\mcD{\mathcal{D}}
\def\mcO{\mathcal{O}}

\def\mcI{\mathcal{I}}
\def\mcL{\mathcal{L}}

\def\reals{{{\rm l} \kern -.15em {\rm R} }}

\usepackage{multirow}

\newtheorem{definition}{Definition}[section]
\newtheorem{theorem}{Theorem}[section]

\newtheorem{lemma}[theorem]{Lemma}

\begin{document}

\title{\huge A Tunably-Accurate Laguerre Petrov-Galerkin Spectral Method for Multi-Term Fractional Differential Equations on the Half Line} 
\author{{Anna Lischke$\,^{\mathrm{a}}$, Mohsen Zayernouri$\,^{\mathrm{b},\mathrm{c}}$, and George Em Karniadakis$\,^{\mathrm{a}}$}\medskip\\
\small $^{\mathrm{a}}$ Division of Applied Mathematics, Brown  University, Providence, RI 02912, USA \vspace*{-0.05cm}\\
\small $^{\mathrm{b}}$Department of Computational Mathematics, Science, and Engineering,  \\ \small Michigan State University, 428 S. Shaw Lane, East Lansing, MI 48824, USA \\
\small $^{\mathrm{c}}$Department of Mechanical Engineering, Michigan State University, 428 S. Shaw Lane,  \\ \small  East Lansing, MI 48824, USA }
\date{}
\maketitle

\begin{abstract}
We present a new tunably-accurate Laguerre Petrov-Galerkin spectral method for solving 
linear multi-term fractional initial value problems with derivative orders at most one and 
constant coefficients on the half line. Our 
method results in a matrix equation of special structure which can be solved in $\mcO(N \log N)$ operations. 
We also take advantage of recurrence relations 
for the generalized associated Laguerre functions (GALFs) in order to derive explicit expressions for the entries of the 
stiffness and mass matrices, which can be factored into the product of a diagonal matrix and a lower-triangular 
Toeplitz matrix. The resulting spectral method is efficient for solving multi-term fractional differential 
equations with arbitrarily many terms. We apply this method to a distributed order differential equation, which is 
approximated by linear multi-term equations through the Gauss-Legendre quadrature rule. We provide numerical 
examples demonstrating the spectral convergence and linear complexity of the method.
\end{abstract}


\section{Introduction}
While numerical methods for fractional differential equations have been investigated for over two 
decades \cite{brunner, delves, podlubny, edwards}, the main difficulty in developing these methods,  
unlike their integer-order counterparts, is the large computational cost arising due to the non-local nature of fractional 
differential operators. For example, in finite difference \cite{podlubny, celik, chen, ding} or finite element methods 
\cite{deng, ervin}, data at all grid points or elements are needed in order to achieve an accurate 
approximation to the fractional derivative at a single grid point or element. This results in methods 
that are significantly more complex in both implementation and computational cost than methods 
for integer-order counterparts.

Recently, spectral methods have been applied to these problems, offering the benefit of more 
natural non-local approximations in addition to high accuracy in the case of smooth solutions. For 
non-smooth solutions with singularity of type $(x-a)^\alpha$ (where $a$ is the left-endpoint of the 
approximation interval), we find that using an approximation of the form $(x-a)^\alpha p(x),$ with $p(x)$ 
a polynomial approximation to the smooth part of the solution, will also lead to numerical 
approximations with a high order of accuracy. Zayernouri and Karniadakis 
derived functions of this type as eigenfunctions of fractional Sturm-Liouville problems on a compact 
interval \cite{zayer1,zayer2, zayernouri2015tempered, Zayernouri14-SIAM-Collocation, Zayernouri14-SIAM-Frac-Advection}. Recently, Khosravian-Arab et al. extended this work to fractional Sturm-Liouville 
problems on the half line and derived the generalized associated Laguerre functions (GALFs) \cite{FSL}. Zhang et al. analyzed spectral methods on the half line for a single-term fractional 
initial value problem using a generalized version of the GALFs \cite{zhang}.

In the literature, there are not many spectral methods for the type of multi-term fractional 
differential equations considered in this work. The existing methods include spectral collocation and 
tau methods \cite{baleanu, bhrawy}, but not Galerkin spectral methods. The linear systems resulting from 
these methods are dense and lead to large condition numbers. In this work, we propose an efficient 
Laguerre Petrov-Galerkin spectral method for multi-term fractional initial value problems (FIVPs) 
on the half line, which leads to sparse and well-conditioned linear systems.

The equations considered in this work are motivated by the approximation of distributed order differential 
equations using a quadrature rule, as in the paper by Diethelm and Ford \cite{diethelm}. This type of equation arises in many physical and biological applications: for example, in applications to viscoelastic oscillators \cite{oscillator}, distributed order membranes in the ear \cite{ear}, 
dielectric induction \cite{caputo1}, and anomalous diffusion \cite{caputo2, sokolov}. In their paper, 
Diethelm and Ford considered distributed order equations of the form
\begin{align}
	\int_0^m a(r) {}_0\mcD_t^r u(t) \ dr &= f(t),
\end{align}
to which they applied the trapezoid quadrature rule to derive a multi-term fractional differential equation 
on a bounded interval. 
To improve the quality of the approximation, many terms in the resulting multi-term equation may be needed. 
The efficiency and high order of accuracy of our proposed method offers the capability of accurately solving 
equations with many terms with low computational cost. 

The multi-term FIVPs considered in this work have fractional order at most one. There is reason to consider 
this an important problem, as it is possible to reduce any linear multi-term fractional equation to a 
system of multi-term fractional equations with order at most one \cite{edwards}.

In the new Petrov-Galerkin method presented in the following sections, we also introduce a tuning 
parameter enabling us to ``speed up" the rate of convergence of the method. Our method 
is also very efficient, as we are able to solve the resulting matrix equation in $\mcO(N \log N)$ operations.

One of the key aspects of our method is the approximation basis functions used, which are the 
eigenfunctions of a fractional singular Sturm-Liouville problem \cite{FSL}. We present some analysis 
that shows how we can use the fractional Sturm-Liouville operator to determine the decay rate of the 
coefficients of the Petrov-Galerkin approximation. We also use recurrence relations for Laguerre polynomials 
to derive explicit expressions for the entries of the stiffness matrices in the Petrov-Galerkin method. This offers 
savings in the cost of assembling these matrices since we avoid using quadrature, in addition to 
avoiding the (potentially large) Gauss-Laguerre quadrature error in stiffness matrix entries.

Another key aspect of the derivation of the Petrov-Galerkin method is fractional integration 
by parts, which we perform in such a way as to offer flexibility in what order of the derivative is 
transferred from the trial function to the test function in the variational form. We demonstrate 
how this flexibility translates into a tunably-accurate method through the derivation of the method and 
with numerical experiments.

The remainder of the paper is organized as follows. Section 2 introduces the multi-term 
fractional initial value problem along with the notation and definitions used throughout the paper. 
In Section 3, we introduce our Petrov-Galerkin spectral method and discuss its computational 
cost. In Section 4, we present numerical examples using fabricated solutions as well as a short 
analysis of the decay rates of the coefficients of the Galerkin projection. In Section 5, we introduce 
distributed order fractional initial value problems as an application of our PG method for multi-term 
equations, with numerical examples in Section 6. Finally, Section 7 offers a summary of our results 
and directions for future research.


\section{Preliminaries}

\subsection{Notation and definitions}

We are interested in solving the multi-term fractional initial value problem (FIVP) with constant 
coefficients $\{b_i\}_{i=1}^K,$ on the interval $t \in (0,+\infty)$:
\begin{align}
\begin{split}
	\sum_{i=1}^K b_i \ {}_0\mcD_t^{\nu_i} u(t) &= f(t), \\
	u(0) &= 0,
\end{split}
\end{align}
where ${}_0\mcD_t^{\nu_i}$ represents the Riemann-Liouville fractional derivative of order 
$\nu_i \in (0,1)$ for all $i = 1,2,\dots,K$. Notice that if the initial condition $u(0)$ is not 
equal to zero, then we can simply apply the same method to solving the modified FIVP
\begin{align}
\begin{split}
	\sum_{i=1}^K b_i \ {}_0\mcD_t^{\nu_i} (u-u_0)(t) &= f(t), \\
	u(0) &= u_0.
\end{split}
\end{align}

\begin{definition} \cite{samko}
	Let $\alpha > 0.$ The left- and right-sided Riemann-Liouville fractional integrals of order $\alpha$ 
	on the semi-infinite interval $(0,+\infty)$ are defined as
	\begin{align}
		\label{def:lsint}
		{}_0\mcI_t^\alpha u(t) &:= \frac{1}{\Gamma(\alpha)} \int_0^t u(s) (t-s)^{\alpha-1} \ ds, \ \ t > 0, \\
		{}_t\mcI_\infty^\alpha u(t) &:= \frac{1}{\Gamma(\alpha)} \int_t^\infty u(s) (s-t)^{\alpha-1} \ ds, \ \ t > 0,
	\end{align}
	where $\Gamma(\cdot)$ denotes the Euler Gamma function.
\end{definition}

\noindent Note that as the right-sided integral is defined on the interval $(t, +\infty),$ $u$ must be a 
function with suitable decay properties as $t \rarrow \infty$ so that this integral is well-defined.

\begin{definition} \cite{samko}
	Let $\nu \in \mbbR_+$ be the order of differentiation on the semi-infinite interval $(0,+\infty)$, and define $m$ such 
	that $m-1 \leq \nu \leq m.$ Then the left- and right-sided Riemann-Liouville derivatives are given by
	\begin{align}
	\label{def:lsrl}
		{}_0\mcD_t^\nu u(t) &= \frac{1}{\Gamma(m-\nu)} \frac{d^m}{dt^m} \int_0^t u(s) (t-s)^{m-\nu-1} \ ds, \ \ \ \ t > 0, \\
		{}_t\mcD_\infty^\nu u(t) &= \frac{1}{\Gamma(m-\nu)} \frac{(-d)^m}{dt^m} \int_t^\infty u(s) (s-t)^{m-\nu-1} \ ds, \ \ \ \ t > 0.
	\end{align}
\end{definition}


\subsection{Fractional Sturm-Liouville problem on the half line}

Following \cite{FSL}, we consider the fractional Sturm-Liouville problem of 
the first kind (FSLP-1) on the half line, and we use the following theorem.
\begin{theorem} \cite{FSL}
The exact eigenfunctions of the following FSLP-1
\begin{align}
\label{FSLP1}
	\mcL_{\alpha,\beta}^1[\phi] &:= {}_t\mcD_\infty^\alpha p_1(t) {}_0\mcD_t^\alpha \phi(t) -\lambda_n^1\omega_1^\beta \phi(t) = 0,
\end{align}
where $\alpha \in (0,1),$ and
\begin{align}
	p_1(t) &= t^{\alpha - \beta}e^{-t}, \ \ \ \omega_1^\beta (t) = t^{-\beta}e^{-t}
\end{align}
subject to the boundary values
\begin{align}
	\phi(0) = 0, \ \ \ {}_t\mcI_\infty^{1-\alpha}(p_1(t) {}_0\mcD_t^\alpha y(t)) \bigg|_{t = \infty} = 0,
\end{align}
are given as
\begin{align}
	\phi_n^{\beta,1}(t) = t^\beta L_n^{(\beta)}(t), \ \ n = 0,1,2, \dots, 
\end{align}
where $\beta > 0$ and the corresponding distinct eigenvalues are
\begin{align}
	\lambda_n^1 &= \frac{\Gamma(n+\beta+1)}{\Gamma(n+\beta-\alpha+1)}, \ \ n = 0,1,2,\dots.
\end{align}
\end{theorem}

We also have from \cite{FSL} the solution to the fractional Sturm-Liouville 
problem of the second kind (FSLP-2) on the half line.
\begin{theorem} \cite{FSL}
The exact eigenfunctions of the following FSLPs-2
\begin{align}
\label{FSLP2}
	\mcL_{\alpha,\beta}^2[\phi] &:= {}_0\mcD_t^\alpha p_2(t) {}_t\mcD_\infty^\alpha \phi(t) - \lambda_n^2 \omega_2^\beta(t) y(t) = 0,
\end{align}
where $\alpha \in (0,1)$ and 
\begin{align}
	p_2(t) &= t^{\beta+\alpha}e^t, \ \ \ \omega_2^\beta(t) = t^\beta e^t,
\end{align}
subject to the boundary values
\begin{align}
	\lim_{t \rarrow +\infty} y(t) = 0, \ \ \ {}_0\mcI_t^{1-\alpha}(p_2(t) {}_t\mcD_\infty^\alpha \phi(t)) \bigg|_{t=0} = 0,
\end{align}
are given as
\begin{align}
	\phi_n^{\beta,2}(t) &= e^{-t}L_n^{(\beta)} (t), \ \ \ n = 0, 1,2,\dots, 
\end{align}
where $\beta > -1$ and the corresponding distinct eigenvalues are
\begin{align}
	\lambda_n^2 &= \frac{\Gamma(n+\beta+\alpha+1)}{\Gamma(n+\beta+1)}, \ \ \ n = 0,1,2,\dots.
\end{align}
\end{theorem}

We will make use of the fact that our trial basis functions are the eigenfunctions of the FSLP-1 
in Section 3.4 below, where we discuss the rate of decay of the coefficients of our Galerkin expansion.


\subsection{Useful properties of Laguerre polynomials}
The left- and right-sided Riemann-Liouville derivatives of the generalized associated Laguerre 
functions (GALFs) are given by (from \cite{FSL})
\begin{align}
	{}_0\mcD_t^{\nu} \phi_m^{\alpha_1,1}(t) &= \frac{\Gamma(m+\alpha_1)}{\Gamma(m+\alpha_1-\nu)} t^{\alpha_1 - \nu} L_{m-1}^{(\alpha_1-\nu)}(t) = \frac{\Gamma(m+\alpha_1)}{\Gamma(m+\alpha_1-\nu)} \phi_m^{\alpha_1-\nu,1}(t), \\
	{}_t\mcD_\infty^\nu \phi_k^{\alpha_2,2}(t) &= e^{-t} L_{k-1}^{(\nu + \alpha_2)}(t) = \phi_k^{\nu + \alpha_2,2}(t),
\end{align}
where $\nu > 0$ and $\alpha_1, \alpha_2 > -1.$ 
\begin{lemma}
The GALFs satisfy the following orthogonality property.
\begin{align}
	\int_0^\infty \phi_n^{\beta,1}(t) \phi_k^{\beta,2}(t) \ dt &= \int_0^\infty t^\beta e^{-t} L_n^{(\beta)}(t) L_k^{(\beta)}(t) \ dt = \gamma_n^\beta \delta_{kn}, \\
\label{gamma}
	\gamma_n^\beta &:= \frac{\Gamma(n+\beta+1)}{\Gamma(n+1)}.
\end{align}
\label{orthogonal}
\end{lemma}
\noindent Notice that when $\beta = 0,$ the resulting matrix is the identity.


\subsection{Fractional integration by parts}

In order to develop the Petrov-Galerkin method, we will need to employ fractional 
integration by parts on the half line involving the GALFs. We will prove Lemma \ref{intbyparts} 
following the technique presented in \cite{Li2009SpaceTime}.
\begin{lemma}
\label{intbyparts}
For real $\nu,$ $0 < \nu < 1,$ if $\Omega := (0,+\infty)$, $\phi_n^{\alpha,1}(t)$ is the GALF of 
the first kind, and $\phi_k^{\beta,2}(t)$ is the GALF of the second kind, and 
$\alpha, \beta > -1,$ then
\begin{align}
	\left({}_0\mcD_t^\nu \phi_n^{\alpha,1}(t), \phi_k^{\beta,2}(t)\right)_\Omega &= \left(\phi_n^{\alpha,1}(t), {}_t\mcD_\infty^\nu \phi_k^{\beta,2}(t)\right)_\Omega.
\end{align}
\end{lemma}
\begin{proof}
Using integration by parts,
\begin{align}
\label{star1}
\begin{split}
	\left({}_0\mcD_t^\nu \phi_n^{\alpha,1}(t), \phi_k^{\beta,2}(t)\right)_\Omega &= \int_0^\infty {}_0\mcD_t^\nu \left\{ t^\alpha L_{n-1}^{(\alpha)}(t) \right\} e^{-t}L_{k-1}^{(\beta)}(t) \ dt \\
	&= \int_0^\infty \frac{1}{\Gamma(1-\nu)} \frac{d}{dt} \int_0^t \frac{s^\alpha L_{n-1}^{(\alpha)}(s)}{(t-s)^\nu} \ ds \ e^{-t} L_{k-1}^{(\beta)}(t) \ dt \\
	&= \frac{e^{-t}L_{k-1}^{(\beta)}(t)}{\Gamma(1-\nu)} \int_0^t \frac{s^\alpha L_{n-1}^{(\alpha)}(s)}{(t-s)^\nu} \ ds \Bigg|_0^\infty - \\
	&\hspace{20pt} \frac{1}{\Gamma(1-\nu)}\int_0^\infty \int_0^t \frac{s^\alpha L_n^{(\alpha)}(s)}{(t-s)^\nu} \ ds \ \frac{d}{dt} \left\{ e^{-t} L_{k-1}^{(\beta)}(t)\right\} \ dt \\
	&= -\frac{1}{\Gamma(1-\nu)} \int_0^\infty \int_0^t \frac{s^\alpha L_{n-1}^{(\alpha)}(s)}{(t-s)^\nu} \ ds \ \frac{d}{dt} \left\{ e^{-t} L_{k-1}^{(\beta)}(t)\right\} \ dt
\end{split}
\end{align}
Now we use integration by parts again.
\begin{align}
\label{star2}
\begin{split}
	\frac{d}{dt} \int_t^\infty \frac{e^{-s} L_{k-1}^{(\beta)}(s)}{(s-t)^\nu} \ ds &= \frac{d}{dt} \left[ \frac{e^{-s}L_{k-1}^{(\beta)}(s)(s-t)^{1-\nu}}{1-\nu}\Bigg|_t^\infty - \frac{1}{1-\nu}\int_t^\infty \frac{d}{ds} \left\{ e^{-s} L_{k-1}^{(\beta)}(s) \right\} (s-t)^{1-\nu} \ ds \right] \\
	&= -\frac{1}{1-\nu}\frac{d}{dt}\int_t^\infty \frac{d}{ds}\left\{e^{-s} L_{k-1}^{(\beta)}(s)\right\} (s-t)^{1-\nu} \ ds \\
	&= \int_t^\infty \frac{\frac{d}{ds} \left\{e^{-s} L_{k-1}^{(\beta)}(s)\right\}}{(s-t)^\nu} \ ds.
\end{split}
\end{align}
Using \eqref{star2}, the right hand side of \eqref{star1} can be written as
\begin{align}
\label{star3}
\begin{split}
	-\frac{1}{\Gamma(1-\nu)} \int_0^\infty & \int_0^t \frac{s^\alpha L_{n-1}^{(\alpha)}(s)}{(t-s)^\nu} \ ds \ \frac{d}{dt} \left\{ e^{-t} L_{k-1}^{(\beta)}(t)\right\} \ dt  \\
	&= -\frac{1}{\Gamma(1-\nu)} \int_0^\infty \int_t^\infty \frac{\frac{d}{ds} \left\{ e^{-s} L_{k-1}^{(\beta)}(s)\right\}}{(s-t)^\nu} \ ds \ t^\alpha L_{n-1}^{(\alpha)}(t) \ dt \\
	&= -\frac{1}{\Gamma(1-\nu)} \int_0^\infty \left(\frac{d}{dt} \int_t^\infty \frac{e^{-s}L_{k-1}^{(\beta)}(s)}{(s-t)^\nu} \ ds \right) t^\alpha L_{n-1}^{(\alpha)}(t) \ dt \hspace{40pt} \text{(by \eqref{star2})} \\
	&= \left(t^\alpha L_{n-1}^{(\alpha)}(t), {}_t\mcD_\infty^\nu \left\{ e^{-t} L_{k-1}^{(\beta)}(t)\right\} \right)_\Omega \\
	&= \left( \phi_n^{\alpha,1}(t), {}_t\mcD_\infty^\nu \phi_k^{\beta,2}(t)\right)_\Omega.
\end{split}
\end{align}
The combination of \eqref{star1} and \eqref{star3} gives the desired result.
\end{proof}

Using the property of Riemann-Liouville fractional derivatives from \cite{podlubny} that if 
$0<p<1,$ $0<q<1,$ $v(0)=0,$ and $t>0,$
\begin{align}
	{}_0\mcD_t^{p+q} v(t) &= {}_0\mcD_t^p \ {}_0\mcD_t^q v(t) = {}_0\mcD_t^q \ {}_0\mcD_t^p v(t),
\end{align}
we can infer from Lemma \ref{intbyparts} that
\begin{align}
\label{intbyparts2}
	\left({}_0\mcD_t^{p+q} \phi_n^{\alpha,1}(t),\phi_k^{\beta,2}(t)\right)_\Omega &= \left({}_0\mcD_t^p \phi_n^{\alpha_1}(t), {}_t\mcD_\infty^q \phi_k^{\beta,2}(t)\right)_\Omega.
\end{align}
We will use property \eqref{intbyparts2} in the variational form for the derivation of our 
Petrov-Galerkin method in the following section.


\section{Petrov-Galerkin spectral method}

As an example problem, we consider the case $K = 2$, with $b_1 = b_2 = 1$:
\begin{align}
\begin{split}
	{}_0\mcD_t^{\nu_1} u(t) + {}_0\mcD_t^{\nu_2} u(t) &= f(t), \hspace{15pt} t \in (0,+\infty), \\
	u(0) &= 0,
\end{split}
\end{align}
where $\nu_1,\nu_2 \in (0,1).$

We use the generalized Laguerre functions to approximate the solution:
\begin{align}
\label{expansion}
	u(t) \approx u_N(t) &= \sum_{n=1}^N a_n \phi_n^{\alpha_1,1}(t),
\end{align}
with $\{a_n\}_{n=1}^N$ the unknown coefficients. The trial and test functions are defined as the eigenfunctions 
of the singular Sturm-Liouville problems of the first and second kinds, respectively:
\begin{align}
	\phi_n^{\alpha_1,1}(t) &= t^{\alpha_1} L_{n-1}^{(\alpha_1)}(t), \\
	\phi_k^{\alpha_2,2}(t) &= e^{-t}L_{k-1}^{(\alpha_2)}(t).
\end{align}

Then the variational form for the PG spectral method is
\begin{align}
\begin{split}
	\sum_{n=1}^N a_n \int_0^\infty \phi_k^{\alpha_2,2}(t) {}_0\mcD_t^{\nu_1} \phi_n^{\alpha_1,1}(t) \ dt + \sum_{n=1}^N a_n \int_0^\infty \phi_k^{\alpha_2,2}(t) {}_0\mcD_t^{\nu_2} \phi_n^{\alpha_1,1}(t) \ dt \\
	= \int_0^\infty f(t) \phi_k^{\alpha_2,2}(t) \ dt =: \hat{f}_k.
\end{split}
\end{align}
Next, we apply Lemma \ref{intbyparts} to the variational form:
\begin{align}
\begin{split}
	\sum_{n=1}^N a_n \int_0^\infty {}_0\mcD_t^{\alpha_1} \phi_n^{\alpha_1,1}(t) {}_t\mcD_\infty^{\nu_1-\alpha_1} \phi_k^{\alpha_2,2}(t) \ dt + \hspace{4cm}\\
	+ \sum_{n=1}^N a_n \int_0^\infty {}_0\mcD_t^{\alpha_1} \phi_n^{\alpha_1,1}(t) {}_t\mcD_\infty^{\nu_2-\alpha_1}\phi_k^{\alpha_2,2}(t) \ dt = \hat{f}_k,
\end{split}
\end{align}
where we keep the left-sided derivative of order $\alpha_1$ applied to the trial basis functions and transfer 
the rest of the derivative to the test functions. We tune $\alpha_1$ to optimize the convergence of the 
spectral method, and $\alpha_2$ is determined by the relation $\alpha_2 = \alpha_1 - \nu_1.$

Using the parameters defined above and Lemma \ref{orthogonal}, the variational form reduces to:
\begin{align}
\begin{split}
	\sum_{n=1}^N a_n & \int_0^\infty {}_0\mcD_t^{\alpha_1} \phi_n^{\alpha_1,1}(t) {}_t\mcD_\infty^{\nu_1 - \alpha_1} \phi_k^{\alpha_1-\nu_1,2}(t) \ dt + \\
	&+ \sum_{n=1}^N a_n \int_0^\infty {}_0\mcD_t^{\alpha_1} \phi_n^{\alpha_1,1}(t) {}_t\mcD_\infty^{\nu_2 - \alpha_1} \phi_k^{\alpha_1-\nu_1,2}(t) \ dt \\
	= \sum_{n=1}^N a_n &\int_0^\infty \frac{\Gamma(n+\alpha_1)}{\Gamma(n)} \phi_n^{0,1}(t) \phi_k^{0,2}(t) \ dt + \\
	&+ \sum_{n=1}^N a_n \int_0^\infty \frac{\Gamma(n+\alpha_1)}{\Gamma(n)} \phi_n^{0,1}(t) \phi_k^{\alpha_1-\nu_1 + \nu_2 - \alpha_1,2}(t) \ dt \\
	= \sum_{n=1}^N a_n & \frac{\Gamma(n+\alpha_1)}{\Gamma(n)} \left[ \delta_{kn} + \int_0^\infty \phi_n^{0,1}(t) \phi_k^{\nu_2 - \nu_1,2}(t) \ dt \right]\\
	= \sum_{n=1}^N a_n & \frac{\Gamma(n+\alpha_1)}{\Gamma(n)} \left[ \delta_{kn} + \int_0^\infty e^{-t} L_{n-1}(t) L_{k-1}^{(\nu_2 - \nu_1)}(t) \ dt\right].
\end{split}
\end{align}

Then it remains to solve the linear system
\begin{align}
	S\vec{a} = \vec{\hat{f}},
\end{align}
where the coefficient matrix $S$ is defined
\begin{align}
\label{def:stiff}
	S_{kn} &= \frac{\Gamma(n+\alpha_1)}{\Gamma(n)}\left[\delta_{kn} + \int_0^\infty \phi_n^{0,1}(t) \phi_k^{\nu_2-\nu_1,2}(t) \ dt \right],
\end{align}
and $\hat{f}_k$ is defined by the integral
\begin{align}
	\hat{f}_k := \int_0^\infty f(t) \phi_k^{\alpha_2,2}(t) \ dt &= \int_0^\infty f(t) e^{-t} L_{k-1}^{(\alpha_2)}(t) \ dt.
\end{align}
We compute this integral using Gauss-Laguerre quadrature.


\subsection{Factorization of the linear system}

The integral in \eqref{def:stiff} has the form
\begin{align}
	Q_{kn} := \int_0^\infty \phi_n^{0,1}(t) \phi_k^{\nu_2-\nu_1,2}(t) \ dt &= \int_0^\infty e^{-t}L_{n-1}(t) L_{k-1}^{(\nu_2-\nu_1)}(t) \ dt.
\end{align}
The matrix $Q$ is a lower-triangular Toeplitz matrix, i.e.
\begin{align}
	Q &= \begin{bmatrix} q_1 & 0 & 0 & \cdots & 0 \\ q_2 & q_1 & 0 & \cdots & 0 \\ q_3 & q_2 & q_1 & \cdots & 0 \\ \vdots & \ddots & \ddots & \ddots & \vdots \\ q_N & q_{N-1} & q_{N-2} & \cdots & q_1 \end{bmatrix}
\end{align}
where the entries are given by the formula
\begin{align}
\label{stiff-formula}
	q_{k-n+1} &= \prod_{i=1}^{k-n} \frac{\nu_2 - \nu_1 + i - 1}{i},
\end{align}
with $k$ as the row index and $n$ as the column index of $Q.$ We can use the 
formula \eqref{stiff-formula} to assemble the stiffness matrix with explicit expressions for each 
entry instead of using quadrature. This will offer significant savings in the cost of assembling the 
stiffness matrix as well as eliminate any approximation error for these entries. The Toeplitz 
structure offers additional savings in storage and makes the process of $p$-refinement efficient 
since we can store the values of the stiffness matrix from the previous approximation. 
Indeed, going from the $N^{\text{th}}$ order expansion to the $(N+1)^{\text{th}}$ 
requires that we add one row and one column to $Q$ (hence $S$), but as this matrix will also 
be Toeplitz and the entries $q_m$ of $Q$ only depend on the orders of the fractional derivatives 
and the number of their diagonal $(m)$, the only new entry that we will need to compute is 
$q_{N+1}$.

We can derive formula \eqref{stiff-formula} using the recurrence identity
\begin{align}
\label{recurrence}
	L_n^{(\alpha)}(t) &= \sum_{i=0}^n \binom{\alpha-\beta + n - i -1}{n-i} L_i^{(\beta)}(t).
\end{align}

Consider again the matrix entry $Q_{kn}$:
\begin{align}
	Q_{kn} &= \int_0^\infty e^{-t} L_{n-1}(t) L_{k-1}^{(\nu_2 - \nu_1)}(t) \ dt.
\end{align}
We plug in the recurrence identity \eqref{recurrence} to expand the Laguerre polynomial 
$L_{k-1}^{(\nu_2-\nu_1)}(t)$ in terms of standard Laguerre polynomials:
\begin{align}
	Q_{kn} &= \int_0^\infty e^{-t} L_{n-1}(t) \sum_{i=0}^{k-1}\binom{\nu_2 - \nu_1 + k-i-2}{k-i-1} L_i(t) \ dt \\
	&= \sum_{i=0}^{k-1} \binom{\nu_2 - \nu_1 + k-i-2}{k-i-1} \int_0^\infty e^{-t} L_{n-1}(t) L_i(t) \ dt \\
	&= \sum_{i=1}^{k} \binom{\nu_2 - \nu_1 + k - i - 1}{k-i} \int_0^\infty e^{-t} L_{n-1}(t) L_{i-1}(t) \ dt \\
	&= \sum_{i=1}^k \binom{\nu_2 - \nu_1 + k - i - 1}{k-i} \delta_{ni} \\
	&= \begin{cases} \binom{\nu_2 - \nu_1 + k - n -1}{k-n}, & n \leq k \\ 0, & n > k. \end{cases}
\end{align}
This implies that $Q$ is lower triangular. Using the product formula to compute the binomial coefficient, we find that
\begin{align}
	Q_{kn} &= \begin{cases} \frac{1}{(k-n)!} \Pi_{\ell=1}^{k-n} (\nu_2 - \nu_1 + \ell - 1), & k \geq n, \\
	0, & k < n. \end{cases}
\end{align}
Hence we can construct the coefficient matrix $S$ exactly, and we will show that $S$ is a lower-triangular 
matrix which can be factored in a way that reduces the complexity of solving the linear system to 
$\mcO(N \log N)$ operations.

To demonstrate how this is done, we define the matrix $\tilde{S}$ by
\begin{align}
	\tilde{S}_{kn} &=\delta_{kn}+ \int_0^\infty e^{-t} L_{n-1}(t) L_{k-1}^{(\nu_2-\nu_1)}(t) \ dt.
\end{align}

Then
\begin{align}
	S_{kn} &= \frac{\Gamma(n+\alpha_1)}{\Gamma(n)} \tilde{S}_{kn},
\end{align}
where $n$ is the column index. If $\{s_n\}_{n=1}^N$ are the column vectors of $S,$  and $\{\tilde{s}_n\}_{n=1}^N$ 
are the column vectors of $\tilde{S},$ then for the solution vector $a = [a_1 \ a_2 \ \cdots \ a_N],$ 
we have
\begin{align}
\begin{split}
	Sa &= \begin{bmatrix} \ \\ \ \\ s_1 \\ \ \\ \ \end{bmatrix} a_1 + \begin{bmatrix} \ \\ \ \\ s_2 \\ \ \\ \ \end{bmatrix}  a_2 + \cdots + \begin{bmatrix} \ \\ \ \\ s_N \\ \ \\ \ \end{bmatrix} a_N \\
	&= \begin{bmatrix} \ \\ \ \\ \tilde{s}_1 \\ \ \\ \ \end{bmatrix} \frac{\Gamma(1+\alpha_1)}{\Gamma(1)} a_1 + \cdots + \begin{bmatrix} \ \\ \ \\ \tilde{s}_N \\ \ \\ \ \end{bmatrix} \frac{\Gamma(N+\alpha_1)}{\Gamma(N)}a_N \\
	&= \begin{bmatrix} \ \\ \ \\ \tilde{s}_1 \\ \ \\ \ \end{bmatrix} \tilde{a}_1 + \cdots + \begin{bmatrix} \ \\ \ \\ \tilde{s}_N \\ \ \\ \ \end{bmatrix} \tilde{s}_N \\
	&= \tilde{S}\tilde{a} = \hat{f},
\end{split}
\end{align}
where $\hat{f}$ is known, $\tilde{S}$ is a lower-triangular Toeplitz matrix, and
\begin{align}
	\tilde{a}_k &:= \frac{\Gamma(k+\alpha_1)}{\Gamma(k)} a_k.
\end{align}
This procedure is equivalent to factoring the stiffness matrix into a Toeplitz matrix $\tilde{S}$ 
and a diagonal matrix $D$, resulting in the linear system with the form
\begin{align}
	Sa = \tilde{S} \tilde{a} = \tilde{S} Da = \hat{f},
\end{align}
where
\begin{align}
	\tilde{S} &= \begin{bmatrix} q_1 + 1 & 0 & 0 & \cdots & 0 \\ q_2 & q_1 + 1 & 0 & \cdots & 0 \\ q_3 & q_2 & q_1 + 1 & \cdots & 0 \\ \vdots & \ddots & \ddots & \ddots & \vdots \\ q_N & q_{N-1} & q_{N-2} & \cdots & q_1 + 1 \end{bmatrix}, \\
	\label{diagonal}
	 \tilde{a}  = Da &= \begin{bmatrix} \frac{\Gamma(1+\alpha_1)}{\Gamma(1)} & 0 & 0 & \cdots & 0 \\ 0 & \frac{\Gamma(2+\alpha_1)}{\Gamma(2)} & 0 & \cdots & 0 \\ 0 & 0 & \frac{\Gamma(3+\alpha_1)}{\Gamma(3)} & \cdots & 0 \\ \vdots & \ddots & \ddots & \ddots & \vdots  \\ 0 & 0 & 0 & \cdots & \frac{\Gamma(N + \alpha_1)}{\Gamma(N)} \end{bmatrix}  \begin{bmatrix} a_1 \\ a_2 \\ a_3 \\ \vdots \\ a_N \end{bmatrix}
\end{align}

Hence we can solve for $\tilde{a}$ in $\mcO(N \log N)$ operations using the algorithm in \cite{toeplitz} and compute $a$ from $\tilde{a}$ in 
another $\mcO(N)$ operations. Since the matrix $D$ only depends on parameter $\alpha_1,$ which 
comes from the approximation $u_N$ itself, the stiffness matrix will have this structure in the case 
where the number of terms in the FIVP, $K$, is greater than 2. In fact, we can solve multi-term FIVPs 
with any number of terms with $\mcO(N \log N)$ operations, as discussed in Section 3.3 below.

It is interesting to note here that mass matrices will have a similar form using this approximation 
method, i.e. when $\nu_i = 0$ for some $i \leq K.$


\subsection{Arbitrary number of terms in the FIVP}

In the above example, we have assumed that $K=2,$ i.e.,
\begin{align*}
	\sum_{i=1}^K b_i \ {}_0\mcD_t^{\nu_i} u(t) &= b_1 {}_0\mcD_t^{\nu_1} u(t) + b_2 {}_0\mcD_t^{\nu_2} u(t) = f(t).
\end{align*}
The next natural question is whether we achieve a similar structure of the stiffness matrix if the 
number of terms on the left hand side, $K$, is greater than 2. If we follow the same derivation of the 
stiffness matrix as above in the case where $K=3$ with $b_1 = b_2 = b_3 = 1,$ for example, we 
find that
\begin{align}
	S_{nk} &= \frac{\Gamma(n+\alpha_1)}{\Gamma(n)} \left[ \delta_{nk} + \int_0^\infty \phi_n^{0,1}(t) \phi_k^{\nu_2-\nu_1,2}(t) \ dt + \int_0^\infty \phi_n^{0,1}(t) \phi_k^{\nu_3-\nu_1,2}(t) \ dt\right]
\end{align}

Then we define matrices $Q_1$ and $Q_2$ as
\begin{align}
\begin{split}
	(Q_1)_{kn} &:= \int_0^\infty \phi_n^{0,1}(t) \phi_k^{\nu_2-\nu_1,2}(t) \ dt = \int_0^\infty e^{-t} L_{n-1}(t) L_{k-1}^{(\nu_2-\nu_1)}(t) \ dt \\
	&=\begin{cases}  \frac{1}{(k-n)!} \prod_{\ell=1}^{k-n}(\nu_2 - \nu_1 + \ell - 1), & k \geq n \\ 0, & k < n, \end{cases} \\
	(Q_2)_{kn} &:= \int_0^\infty \phi_n^{0,1}(t) \phi_k^{\nu_3-\nu_1,2}(t) \ dt = \int_0^\infty e^{-t}L_{n-1}(t) L_{k-1}^{(\nu_3-\nu_1)}(t) \ dt \\
	&= \begin{cases} \frac{1}{(k-n)!} \prod_{\ell=1}^{k-n}(\nu_3 - \nu_1 + \ell - 1), & k \geq n \\ 0, & k < n.\end{cases}
\end{split}
\end{align}
If we represent the diagonal entries of $Q_1$ and $Q_2$ by $q_m^{(1)}$ 
and $q_m^{(2)},$ respectively, with $m = k-n+1$, the resulting stiffness matrix is
\begin{align}
\begin{split}
	S_{kn} &= \frac{\Gamma(n+\alpha_1)}{\Gamma(n)} \tilde{S}_{kn}, \\
	\tilde{S}_{kn} &= \begin{bmatrix} q_1^{(1)} + q_1^{(2)} + 1 & 0 & 0 & \cdots & 0 \\ q_2^{(1)} + q_2^{(2)} & q_1^{(1)} + q_1^{(2)} + 1 & 0 & \cdots & 0 \\ q_3^{(1)} + q_3^{(2)} & q_2^{(1)} + q_2^{(2)} & q_1^{(1)} + q_1^{(2)} + 1 & \cdots & 0 \\ \vdots & \ddots & \ddots & \ddots & \vdots \\ q_N^{(1)} + q_N^{(2)} & q_{N-1}^{(1)} + q_{N-1}^{(2)} & \cdots & \cdots & q_1^{(1)} + q_1^{(2)} + 1 \end{bmatrix}
\end{split}
\end{align}
Hence $S$ can again be factored into $\tilde{S}D$ with $D$ defined as in \eqref{diagonal}.
Then we define $\tilde{a}_k := \frac{\Gamma(k+\alpha_1)}{\Gamma(k)} a_k$ in the same way as 
before, and follow the same procedure as in the $K=2$ case to invert $\tilde{S}$ and $D$. We continue in this way for any value of $K \in \mbbN$ to 
see that we can solve the resulting linear system for any number of terms using $\mcO(N \log N)$ operations.


\subsection{Spectral decay of coefficients in Galerkin projection}

In this section, we are mainly interested in the rate of decay of the coefficients of the Galerkin expansion. Given 
the weight function $w(t) = t^{-\beta}e^{-t},$ we expand a function $u(t) \in L^2_w(0,\infty)$ by
\begin{align}
	u(t) &\approx u_N(t) := \sum_{n=1}^N a_n t^\beta L_{n-1}^\beta(t).
\end{align}
Then following \cite{shen}, since $t^\beta L_{n-1}^\beta(t)$ is an eigenfunction for the FSLP-1, 
we have from \eqref{gamma} and Lemma \ref{intbyparts},
\begin{align}
	\|u\|_{L^2_w}^2 &= \sum_{n=1}^N \gamma_n^\beta |a_n|^2.
\end{align}
Therefore
\begin{align}
\begin{split}
	a_n &= \frac{1}{\gamma_n^\beta}(u, \phi_n^{\beta,1})_{L^2_w} \\
	&= \frac{1}{\gamma_n^\beta} \int_0^\infty u(t) \phi_n^{\beta,1}(t) w(t) \ dt \\
	&= \frac{1}{\gamma_n^\beta \lambda_n^1} \int_0^\infty u(t) \mcL_1^{\alpha,\beta}[\phi_n^{\beta,1}(t)] \ dt \\
	&= \frac{1}{\gamma_n^\beta \lambda_n^1} \int_0^\infty \left({}_0\mcD_t^\alpha u(t)\right) e^{-t} t^{\alpha-\beta} {}_0\mcD_t^\alpha \phi_n^{\beta,1}(t) \ dt \\
	&= \frac{1}{\gamma_n^\beta \lambda_n^1} \int_0^\infty \mcL_1^{\alpha,\beta}[u(t)]\phi_n^{\beta,1}(t) \ dt \\
	&= \frac{1}{\gamma_n^\beta \lambda_n^1}(u_{(1)},\phi_n^{\beta,1})_{L^2_w}.
\end{split}
\end{align}
We have defined $u_{(m)}$ as in \cite{hesthaven}
\begin{align}
	u_{(m)}(t) &= \frac{1}{w(t)} \mcL u_{(m-1)}(t) = \left(\frac{\mcL}{w(t)}\right)^m u(t).
\end{align}

Then 
\begin{align}
\begin{split}
	\frac{1}{\gamma_n^\beta \lambda_n^1} (u_{(1)},\phi_n^{\beta,1})_{L^2_w} &= \frac{1}{\gamma_n^\beta (\lambda_n^1)^2} (u_{(2)}, \phi_n^{\beta,1})_{L^2_w} \\
	&= \cdots\\
	&= \frac{1}{\gamma_n^\beta (\lambda_n^1)^m} (u_{(m)},\phi_n^{\beta,1})_{L^2_w}.
\end{split}
\end{align}

We know from \cite{FSL} that the eigenvalues have the asymptotic 
similarity
\begin{align}
	\lambda_n \sim n^\alpha.
\end{align}

So the coefficients of the approximation decay at the rate:
\begin{align}
	|a_n| &\simeq C\frac{1}{(\lambda_n^1)^m} \|u_{(m)}\|_{L^2_w} \sim C n^{-\alpha m} \|u_{(m)}\|_{L^2_w}.
\end{align}

If $u \in C^\infty(0,\infty),$ we expect exponential convergence of the approximation.


\section{Numerical Results}

In this section, we present numerical examples which demonstrate the validity of our proposed 
method. We plot relative errors computed using Gauss Laguerre quadrature 
for various values of $N,$ which represents the number of terms in the Galerkin expansion. The 
formula for the relative errors, represented by $e_N,$ is given by
\begin{align}
	e_N &= \frac{\|u^{\text{ext}} - u_N\|_{\omega,L^2(0,\infty)}}{\|u^{\text{ext}}\|_{\omega,L^2(0,\infty)}},
\end{align}
where the weight function is $\omega(t) = e^{-t}.$

\subsection{Example 1.}

In this example, we solve the multi-term FIVP
\begin{align}
\begin{split}
	{}_0\mcD_t^{1/3} u(t) + {}_0\mcD_t^{1/2} u(t) &= f(t), \\
	u(0) &= 0.
\end{split}
\end{align}
We test the method using the fabricated solution $u^{\text{ext}}(t) = t^{3+1/4}.$

In Figure \ref{example1}, we plot the numerical solutions using seven different values of the 
tuning parameter $\alpha_1.$ Recall that the basis functions used in the Galerkin expansion for 
this method have the form
\begin{align}
	\phi_n^{\alpha_1,1}(t) &= t^{\alpha_1}L_n^{(\alpha_1)}(t),
\end{align}
so adjusting this tunable parameter requires an entirely new approximation. Since the fabricated 
solution has a fractional singularity of order $1/4,$ we expect that the method will return the 
exact solution when $\alpha_1 = 1/4$. We can see that this is 
consistent with Figure \ref{example1}, where the relative errors corresponding to these values 
of $\alpha_1$ drop to machine precision after three and four terms are used in the expansion, 
respectively.

We achieve algebraic convergence in this example, since the solution has finite regularity. 
The rates of convergence printed in the legend of Figure \ref{example1} are computed 
by taking the slope in the log-log scale of the line between the last two computed relative errors. 
In view of the regularity of the fabricated solution, the results in Figure \ref{example1} demonstrate 
that the method converges optimally for this example.

Further, the tunable accuracy of the method is demonstrated in that the smallest perturbation 
from the optimal $\alpha_1$-values results in the fastest rate of convergence (apart from the 
case where the solution is achieved exactly).

\FloatBarrier

\begin{figure}[ht!]
\begin{center}
	\includegraphics[width=11cm]{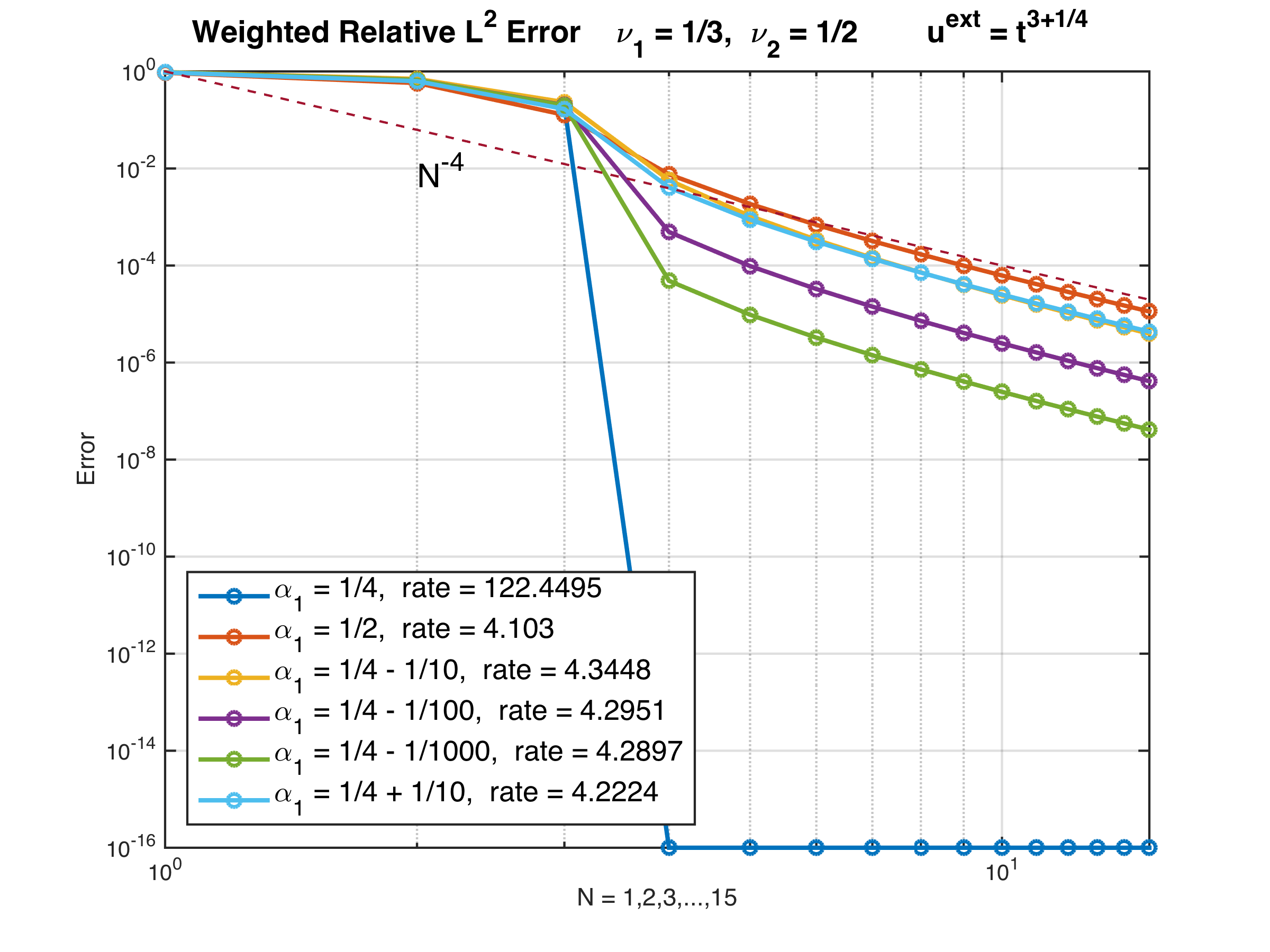}
\end{center}

\caption{Weighted relative $L^2$-error for Example 1 in a log-log scale.}
\label{example1}
\end{figure}

\FloatBarrier

\subsection{Example 2.}

In the next example, we solve the multi-term FIVP
\begin{align}
\begin{split}
	{}_0\mcD_t^{2/3} u(t) + {}_0\mcD_t^{1/10} u(t) &= f(t), \\
	u(0) &= 0.
\end{split}
\end{align}

In this case, we use the fabricated solution $u^{\text{ext}}(t) = t^{5+1/2}.$ The purpose of 
this example is to further assure that the method achieves or exceeds optimal convergence 
rates for any value of $\alpha_1$ given that the fabricated solution is not very smooth. In Example 
2, the highest order derivative of $u^{\text{ext}}(t)$ is order five, but we see that even for 
$\alpha_1$ far away from its optimal value ($\alpha_1 = 1/2$), the convergence rate exceeds seven (e.g., $\alpha_1 = 1/10$).

\FloatBarrier

\begin{figure}[ht!]
\begin{center}
	\includegraphics[width=11cm]{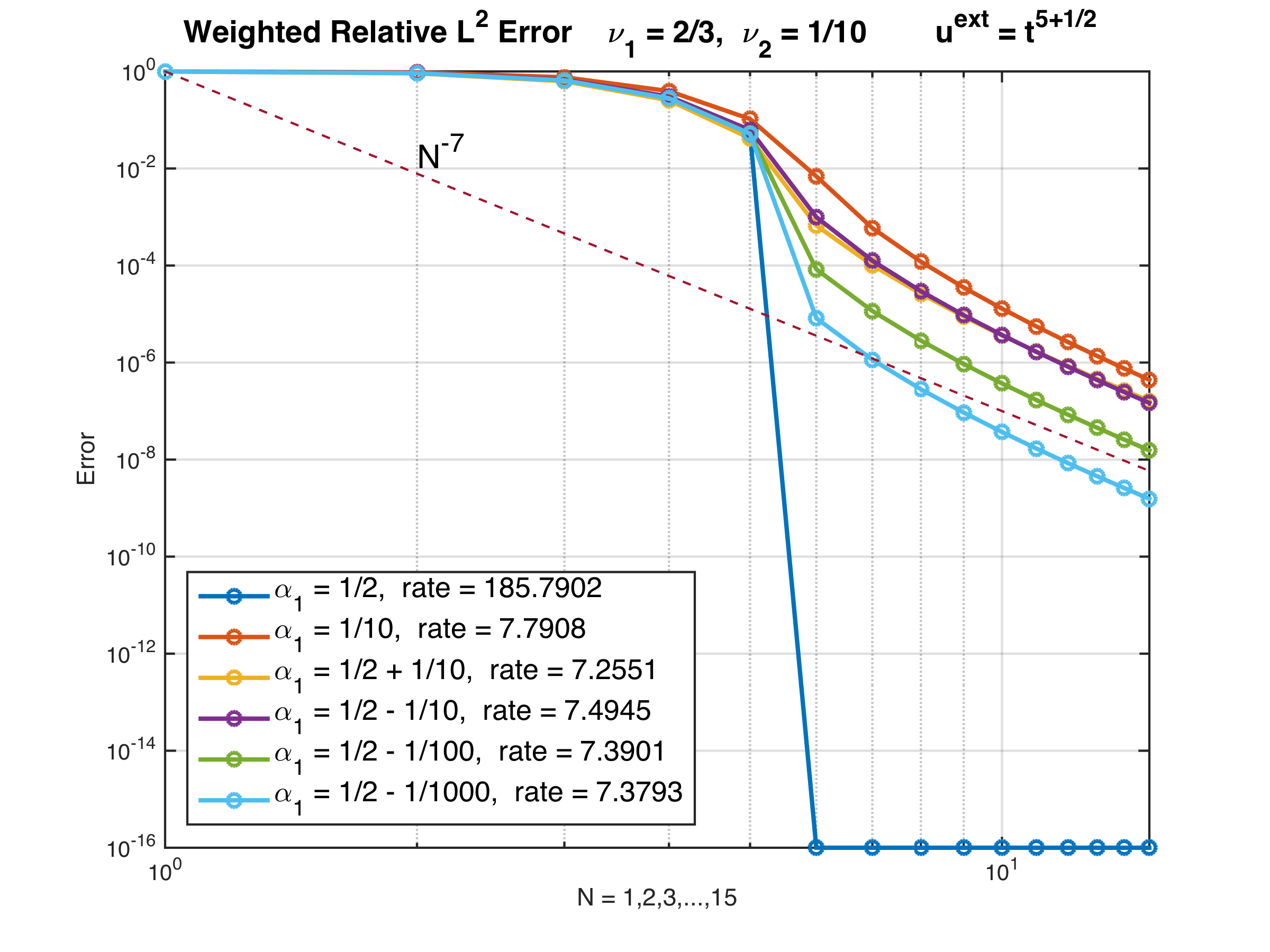}
\end{center}

\caption{Weighted Relative $L^2$-error for Example 2 in a log-log scale.}
\label{example2}
\end{figure}

\FloatBarrier

\subsection{Example 3.}

We again solve a two-term FIVP:
\begin{align}
\begin{split}
	{}_0\mcD_t^{1/4} u(t) + {}_0\mcD_t^{1/5} u(t) &= f(t), \hspace{15pt} t \in (0,+\infty), \\
	u(0) &= 0.
\end{split}
\end{align}
We use the fabricated solution $u^{\text{ext}}(t) = t^{1/2}\sin(t).$

We observe exponential convergence of the method for this example as shown in 
Figure \ref{example3}. In this case, the numerical results for different values of $\alpha_1$ 
are not so different from each other as in the previous examples. The error from choosing a sub-optimal 
$\alpha_1$ value is dominated by the error in approximating the sine function.

\FloatBarrier

\begin{figure}[ht!]
\begin{center}
	\includegraphics[width=11cm]{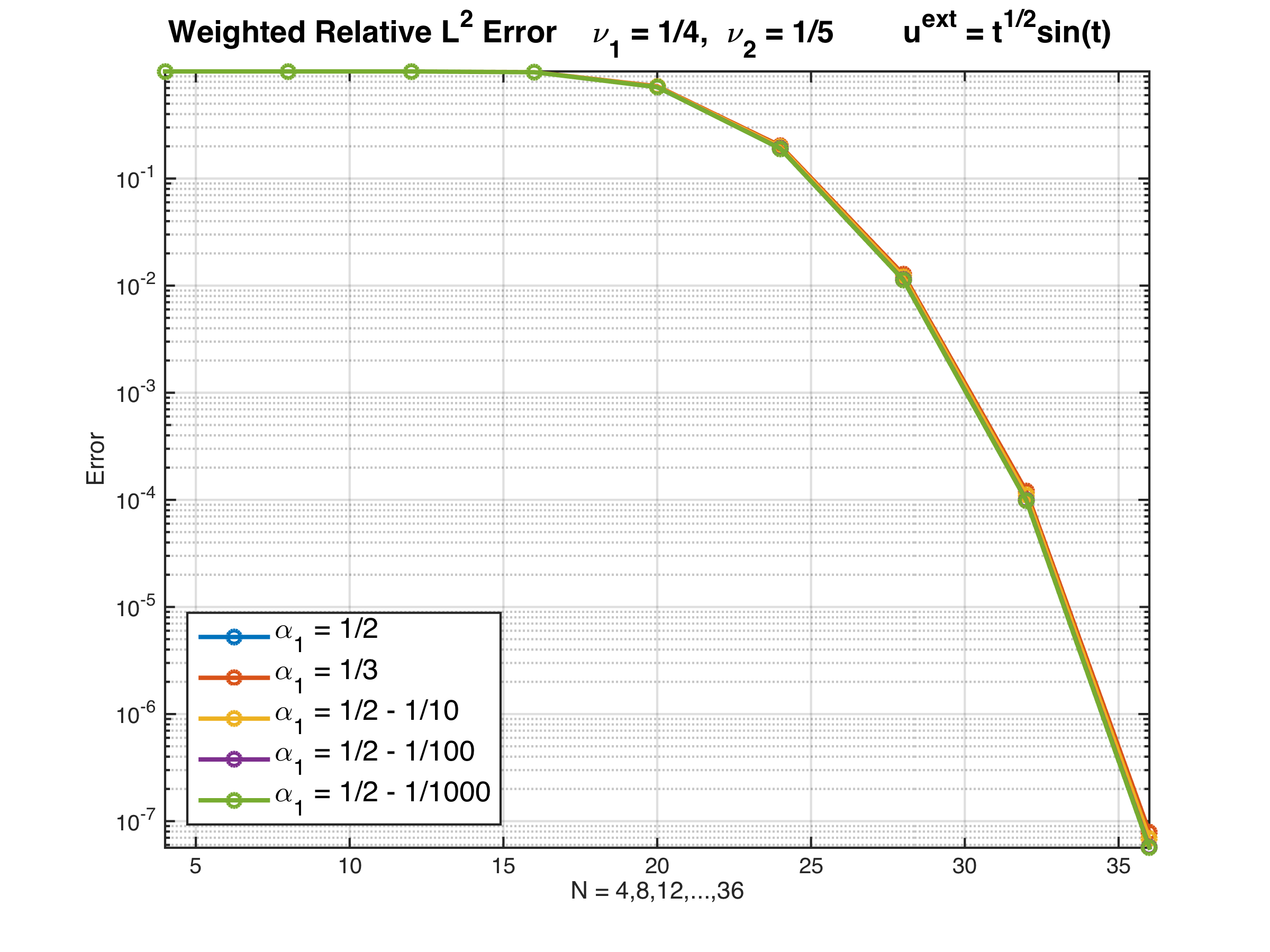}
\end{center}

\caption{Weighted relative $L^2$-error for Example 3 in a log-linear scale.}
\label{example3}
\end{figure}

\FloatBarrier

\subsection{Example 4.}

In Example 4, we solve the two-term FIVP
\begin{align}
\begin{split}
	{}_0\mcD_t^{4/5} u(t) + {}_0\mcD_t^{1/2} u(t) &= f(t), \hspace{15pt} t \in (0,+\infty)\\
	u(0) &= 0.
\end{split}
\end{align}
We use the fabricated solution $u^{\text{ext}}(t) = 5t^{7/2} + 4t^2 + t^{5/3}.$ We believe 
this to be an interesting example because the optimal value of $\alpha_1$ is not clear. Using 
our set of basis functions to approximate this solution will not allow us to capture the result 
exactly in only a few terms as before, since there are two terms with different order fractional 
singularities at $t = 0.$

As shown in Figure \ref{example4}, the approximation using $\alpha_1 = 1/2$ seems to give 
the best approximation to the fabricated solution after the first few values of $N$, although the 
asymptotic convergence rate is slower than for the other tested values. The $\alpha_1$ with the 
fastest convergence rate of those tested is $\alpha_1 = 1/6$.

\FloatBarrier

\begin{figure}[ht!]
\begin{center}
	\includegraphics[width=11cm]{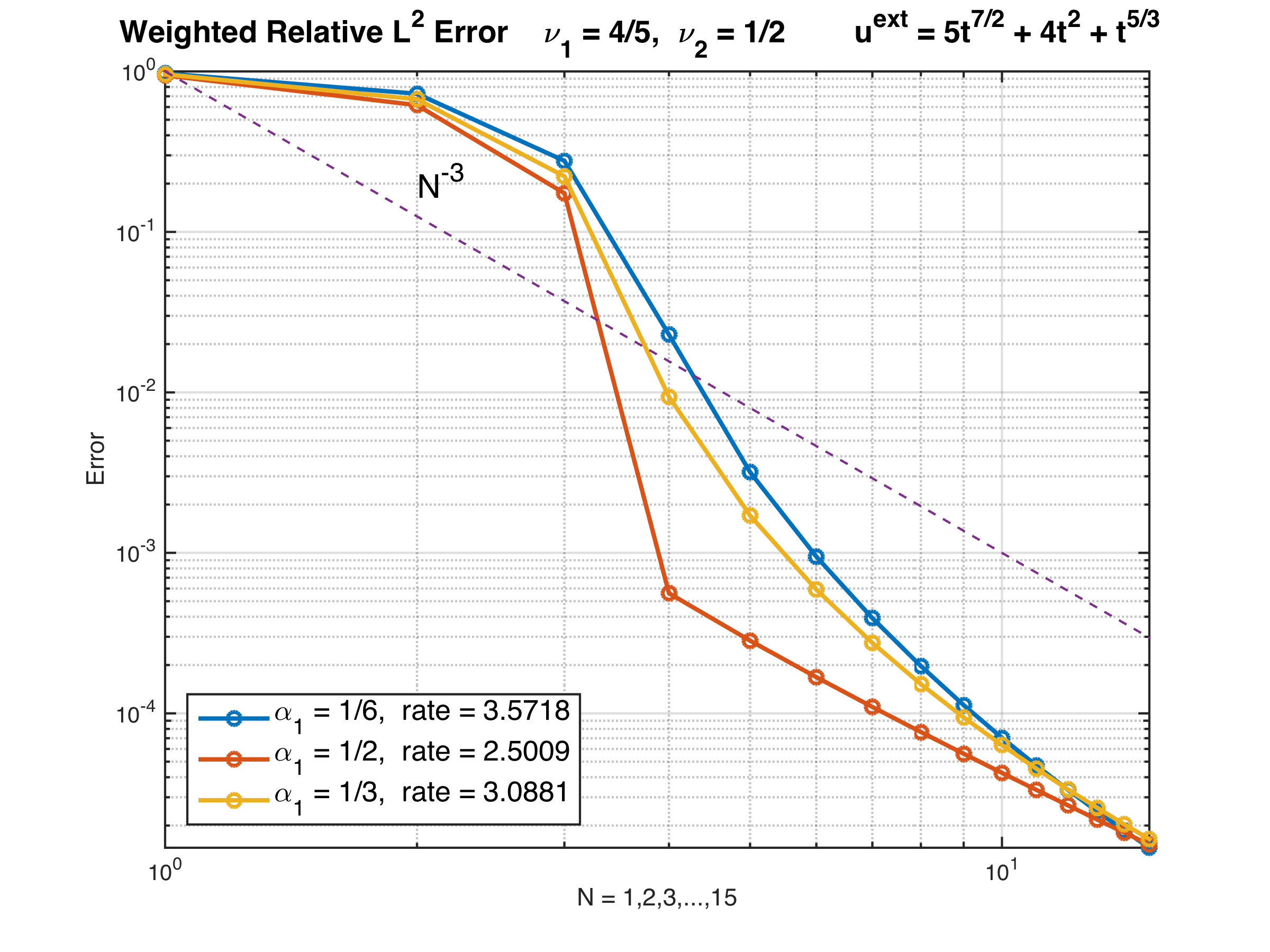}
\end{center}

\caption{Weighted Relative $L^2$-error for Example 4 in a log-log scale.}
\label{example4}
\end{figure}

\FloatBarrier

\subsection{Example 5.}
To demonstrate that we can also solve equations with a larger number of terms with high accuracy, we solve the fifty-term FIVP:
\begin{align}
\begin{split}
	\sum_{i=1}^{50} {}_0\mcD_t^{\nu_i}u(t) &= f(t), \hspace{15pt} t \in (0,+\infty)\\
	u(0) &= 0,
\end{split}
\end{align}
where each $\nu_i \in [0,m],$ with $m \leq 1.$ In this case, 
\begin{align}
\label{orders}
	\nu_i = \frac{(i-1) m}{K-1}, \hspace{15pt} K = 50, \hspace{15pt} m = \frac{11}{12}.
\end{align}
 We use the fabricated solution $u^{\text{ext}}(t) = t^{2+1/4}$ to plot the 
weighted relative $L^2$ error in Figure \ref{fig:example5}

\FloatBarrier

\begin{figure}[ht!]
\begin{center}
	\includegraphics[width=11cm]{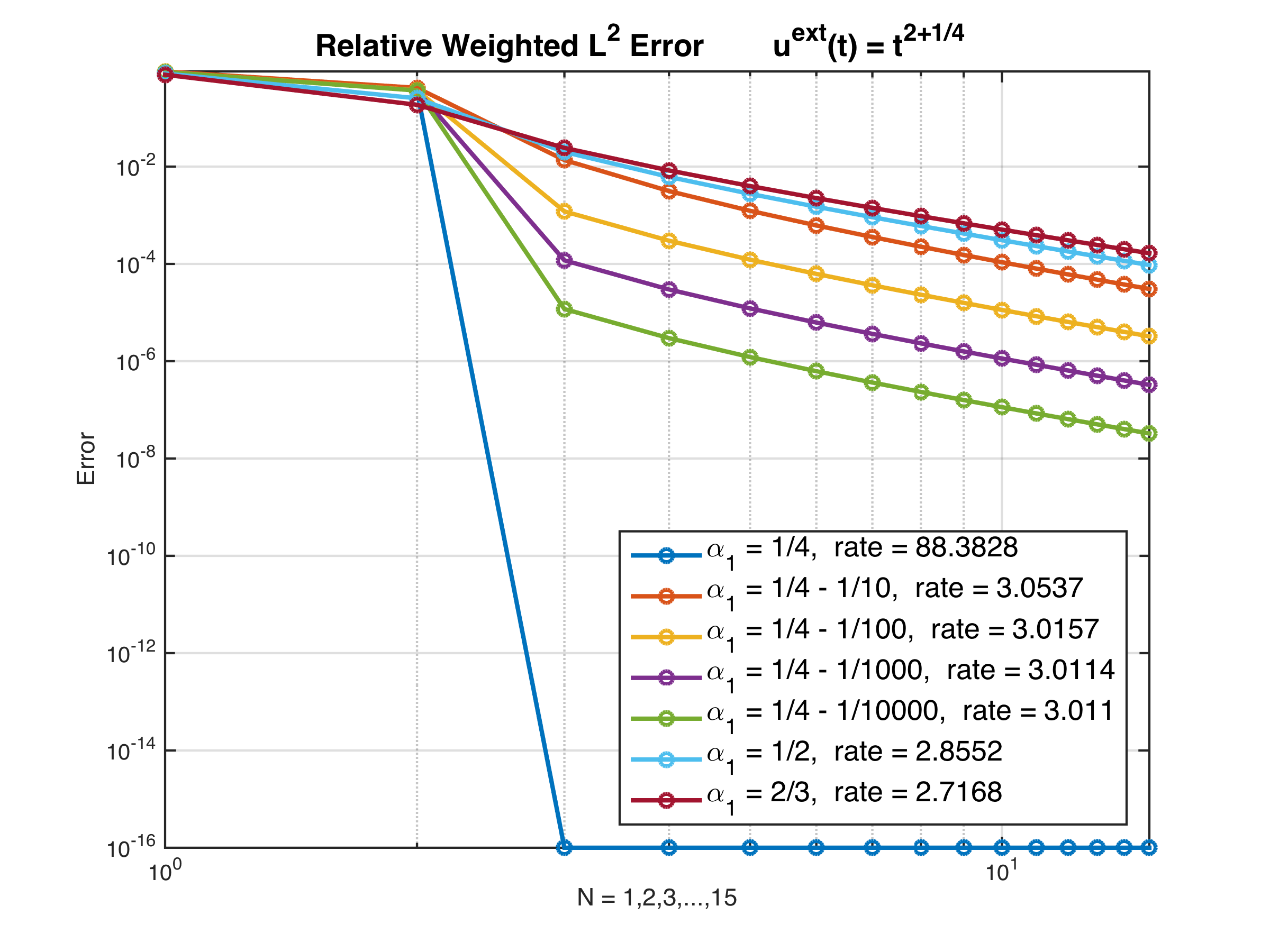}
\end{center}

\caption{Weighted Relative $L^2$-error for Example 5 in a log-log scale.}
\label{fig:example5}
\end{figure}

\FloatBarrier

For this example, we also computed the condition numbers of the stiffness matrices resulting from 
the different values of $\alpha_1.$ We observe that the condition numbers all grow at a rate slower 
than $N.$

\FloatBarrier

\begin{center}
\begin{tabular}{| c | c | c | c | c | c |}
	\hline
	$N$ & $\alpha_1 = \frac{1}{4}$ & $\alpha_1 = \frac{1}{4}-\frac{1}{10}$ & $\alpha_1 = \frac{1}{4} - \frac{1}{100}$ & $\alpha_1 = \frac{1}{2}$ & $\alpha_1 = \frac{2}{3}$ \\
	\hline
	\hline
	2 & 1.5886  &  1.5620  &  1.5849  &  1.7280  &  1.8531 \\
	4 & 2.4325  &  2.2840   & 2.4152  &  2.9990  &  3.4963 \\
	6 & 3.2292  &  2.9345  &  3.1958  &  4.3119 &   5.2943 \\
	8 & 3.9999  &  3.5478  &  3.9490  &  5.6683  &  7.2349 \\
	10 & 4.7533  &  4.1354  &  4.6838 &   7.0639  &  9.3016 \\
	12 & 5.4944  &  4.7040  &  5.4053 &   8.4953 &  11.4816 \\
	14 & 6.2260  &  5.2576  &  6.1166  &  9.9595 &  13.7654 \\
	\hline
\end{tabular}
\captionof{table}{Condition numbers of the stiffness matrices $S$ in the fifty-term equation for different values of the tuning parameter $\alpha_1$.}
\end{center}

\FloatBarrier

In order to compare timings of the method for different values of $K,$ we timed our PG method solving the 
equation in Example 5 for $K = 2, 10,$ and $50$, where the orders $\nu_i$ are defined using the formula in 
\eqref{orders}. In Figure \ref{fig:timings}, we show the timings in actual seconds for $N = 1,2,3,...,30,$ along 
with a best-fit line. The timings include the computation of the load vector $\vec{\hat{f}}$ and inverting the 
linear system to solve for the coefficients $\vec{a}.$ As $N$ increases, we also increase the number of 
quadrature points used for computing $\vec{\hat{f}}$ to maintain the desired level of accuracy. These 
timings were collected with Mathematica using a 3 GHz Intel Core i7 processor.

\FloatBarrier

\begin{figure}[ht!]
\centering
\begin{subfigure}{0.49\textwidth}
\centering
\includegraphics[width = 6.5cm]{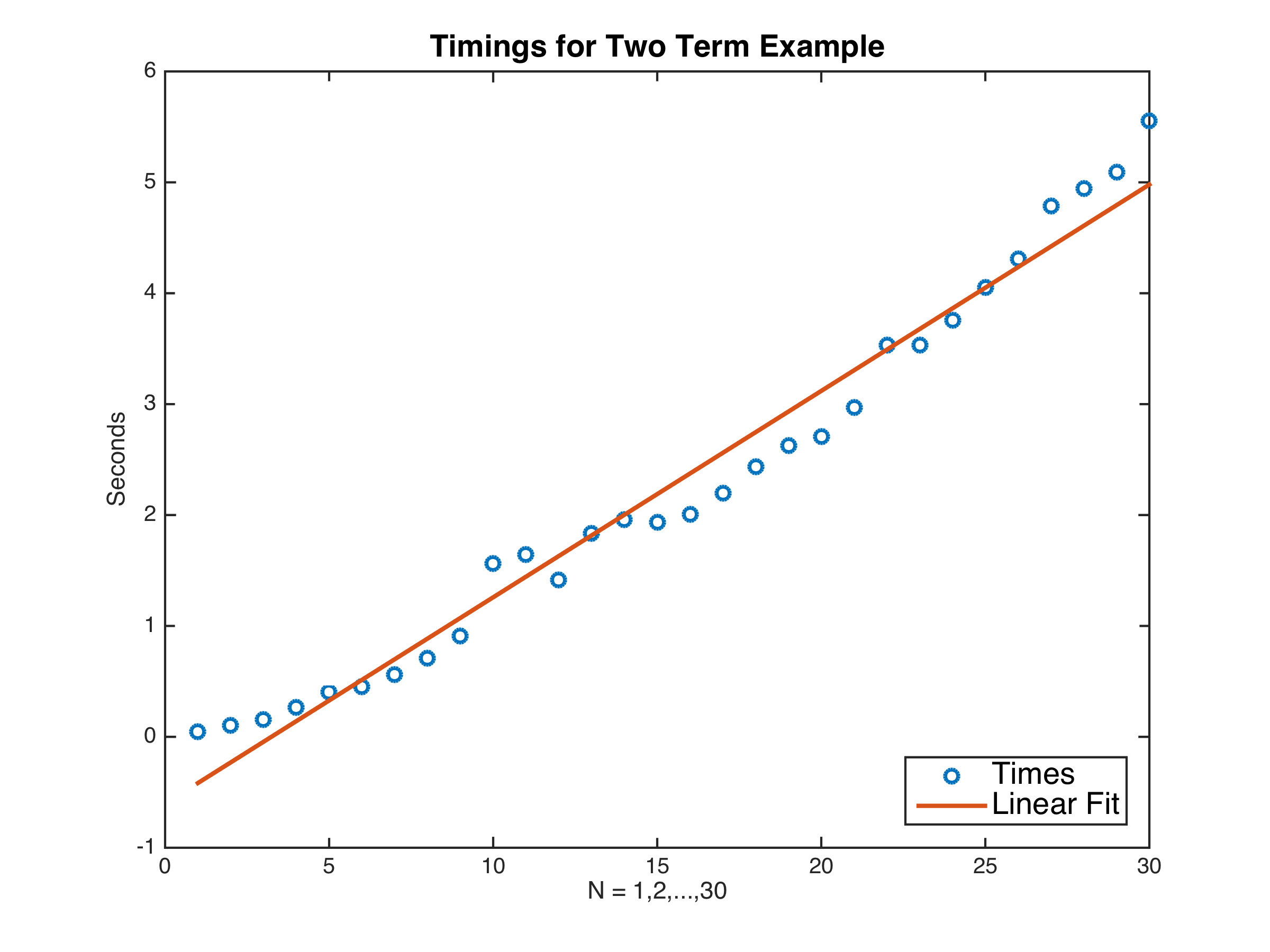}
\caption{Two-term equation.}
\label{fig:left}
\end{subfigure}
\begin{subfigure}{0.49\textwidth}
\centering
\includegraphics[width = 6.5cm]{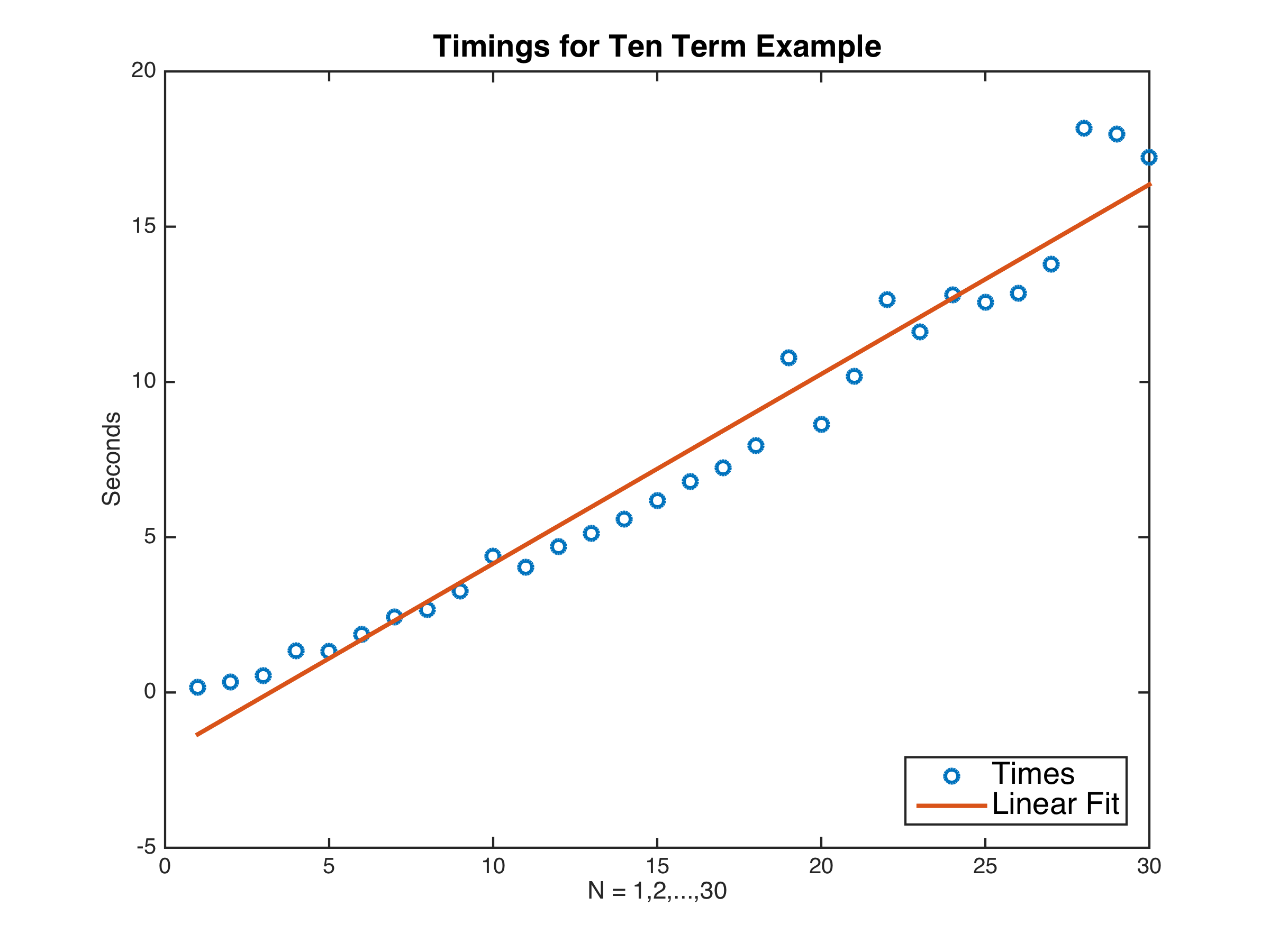}
\caption{Ten-term equation.}
\label{fig:right}
\end{subfigure}
\begin{subfigure}{0.49\textwidth}
\centering
\includegraphics[width = 6.5cm]{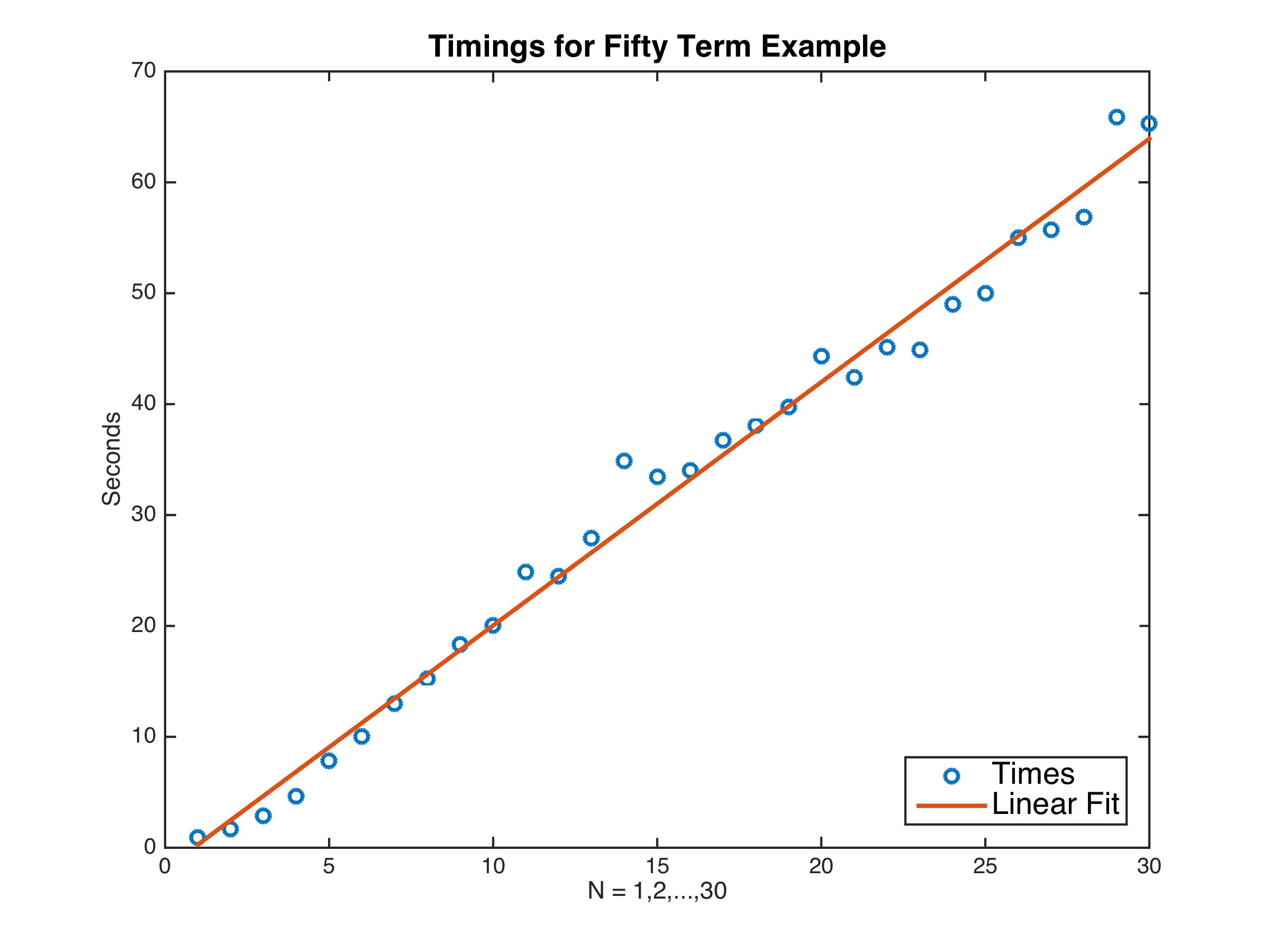}
\caption{Fifty-term equation.}
\label{fig:bottom}
\end{subfigure}
\caption{Timings in actual seconds for the equation from Example 5 with (a) $K = 2$, (b) $K = 10$, and (c) $K = 50$.}
\label{fig:timings}
\end{figure}


\section{Application to distributed order equations}

Multi-term fractional differential equations have been used in combination with a quadrature rule to solve 
distributed order differential equations of the form
\begin{align}
\label{disteqn}
	\int_0^m g(r) {}_0\mcD_t^r u(t) \ dr &= f(t),
\end{align}
where the integral on the left hand side is called the \textit{distributed order derivative}. The function $g(r)$ 
that appears in the integrand is a distribution where the argument $r$ corresponds to the order of the fractional 
derivative. This function must be integrable on $[0,m]$ and satisfy the property $g(r) \geq 0$ for all $r \in [0,m]$.

The idea for solving this equation using multi-term fractional differential equations was proposed by Diethelm 
and Ford \cite{diethelm}, where they applied trapezoidal quadrature to the integral in \eqref{disteqn} to derive  
a linear multi-term equation in a bounded interval with constant coefficients, and then applied a finite difference method 
to solve the distributed order equation.
This application highlights the usefulness of algorithms, which can efficiently solve multi-term equations with 
a high number of terms, as may be necessary to decrease the error due to the quadrature. 

{We observed the convergence rate of the trapezoid rule to be much slower than that of Gauss-Legendre 
quadrature, which has been shown to be spectrally accurate for this setting in the paper by Kharazmi et al. \cite{kharazmi}. 
This is shown in Figures \ref{fig:example6a}, \ref{fig:example6b}, and \ref{fig:example6c}, while in subsequent examples, 
we only show the error plots using Gauss-Legendre quadrature.

\subsection{Numerical method}

We are interested in solving the distributed order fractional differential equation on the half line:
\begin{align}
\label{disteqn2}
\begin{split}
	\int_0^m g(r) {}_0\mcD_t^r u(t) \ dr &= f(t), \hspace{10pt} t \in (0, +\infty) \\
	u(0) &= 0,
\end{split}
\end{align}
where $m \in [0,1]$ and ${}_0\mcD_t^r [\cdot]$ represents a Riemann-Liouville fractional derivative.

We apply Gauss-Legendre quadrature to the left hand side of \eqref{disteqn2} side to get the multi-term FIVP:
\begin{align}
\begin{split}
	\sum_{i=1}^K w_i g(\nu_i) {}_0\mcD_t^{\nu_i}u(t) &\approx f(t), \hspace{10pt} t \in (0,+\infty) \\
	u(0) &= 0,
\end{split}
\end{align}
where $K$ is the number of quadrature nodes $\{\nu_i\}$ and the weights of the quadrature rule are represented by 
$\{w_i\}_{i=1}^K$.
Recall that we approximate the solution to the multi-term equation as
\begin{align}
	u(t) &\approx u_N(t) = \sum_{n=1}^N a_n \phi_n^{\alpha_1,1}(t),
\end{align}
where
\begin{align}
	\phi_n^{\alpha_1,1}(t) &:= t^{\alpha_1} L_{n-1}^{(\alpha_1)}(t),
\end{align}
where $L_{n-1}^{(\alpha_1)}(t)$ is the associated Laguerre polynomial of order $n-1.$

We integrate against the test functions
\begin{align}
	\phi_k^{\alpha_2,2}(t) &:= e^{-t} L_{k-1}^{(\alpha_2)}(t)
\end{align}
where $\alpha_2 = \alpha_1 - \nu_1.$ Then the variational form for the Petrov-Galerkin method is given by
\begin{align}
\begin{split}
	\int_0^\infty \phi_k^{\alpha_2,2}(t) \sum_{i=1}^K w_i g(\nu_i) {}_0\mcD_t^{\nu_i} \left(\sum_{n=1}^N a_n \phi_n^{\alpha_1,1}(t) \right) dr dt &= \int_0^\infty f(t) \phi_k^{\alpha_2,2}(t) dt \\
	&=: \hat{f}_k.
\end{split}
\end{align}
Next, we apply fractional integration by parts and the properties of the GALFs as described above:
\begin{align}
	\sum_{n=1}^N a_n \frac{\Gamma(n+\alpha_1)}{\Gamma(n)}  \left[ w_1 g(\nu_1) \delta_{kn} + \sum_{i=2}^K w_i g(\nu_i) \int_0^\infty e^{-t} L_{n-1}(t) L_{k-1}^{(\nu_i - \nu_1)}(t) \ dt \right] &= \hat{f}_k.
\end{align}

It remains to solve the linear system
\begin{align}
	S\vec{a} = \vec{\hat{f}}
\end{align}
for the vector of coefficients $\vec{a}$ using the factorization methods as described above, where the 
stiffness matrix $S$ is given by
\begin{align}
	S_{kn} = w_1g(\nu_1) \delta_{kn} + \sum_{i=2}^K w_i g(\nu_i) \int_0^\infty e^{-t}L_{n-1}(t) L_{k-1}^{(\nu_i-\nu_1)}(t) \ dt.
\end{align}

\section{Numerical results for distributed order equations}

We present convergence results of our PG method and Gauss-Legendre quadrature applied to the 
distributed order equation \eqref{disteqn2}. The distribution functions $g(r)$ are chosen to be smooth on the 
interval $[0,m]$ where $m < 1.$

\subsection{Example 6.}

In this example, we choose the fabricated solution to be the smooth function $u^{\text{ext}}(t) = t^5$ and 
the distribution function to be $g(r) = \frac{\Gamma(6-r)}{5!}.$ Given these choices, we find that the 
right hand side function $f(t)$ is
\begin{align}
	\int_0^m g(r) {}_0\mcD_t^r u^{\text{ext}}(t) \ dr &= \frac{t^5 - t^{5-m}}{\log(t)} =: f(t).
\end{align}

We can see from the plateaus in the error in Figure \ref{fig:example6a} that the Gauss-Legendre rule 
gives us much faster convergence, as we nearly reach machine precision with $K = 10$ 
quadrature points, as opposed to approximation error of order $\mcO(10^{-3})$ with $K = 50$ quadrature 
points when using the trapezoid rule. We choose the tuning parameter for the PG method to be $\alpha_1 = 1.$

\FloatBarrier

\begin{figure}[ht!]
\centering
\begin{subfigure}{0.49\textwidth}
\centering
\includegraphics[width = 6.5cm]{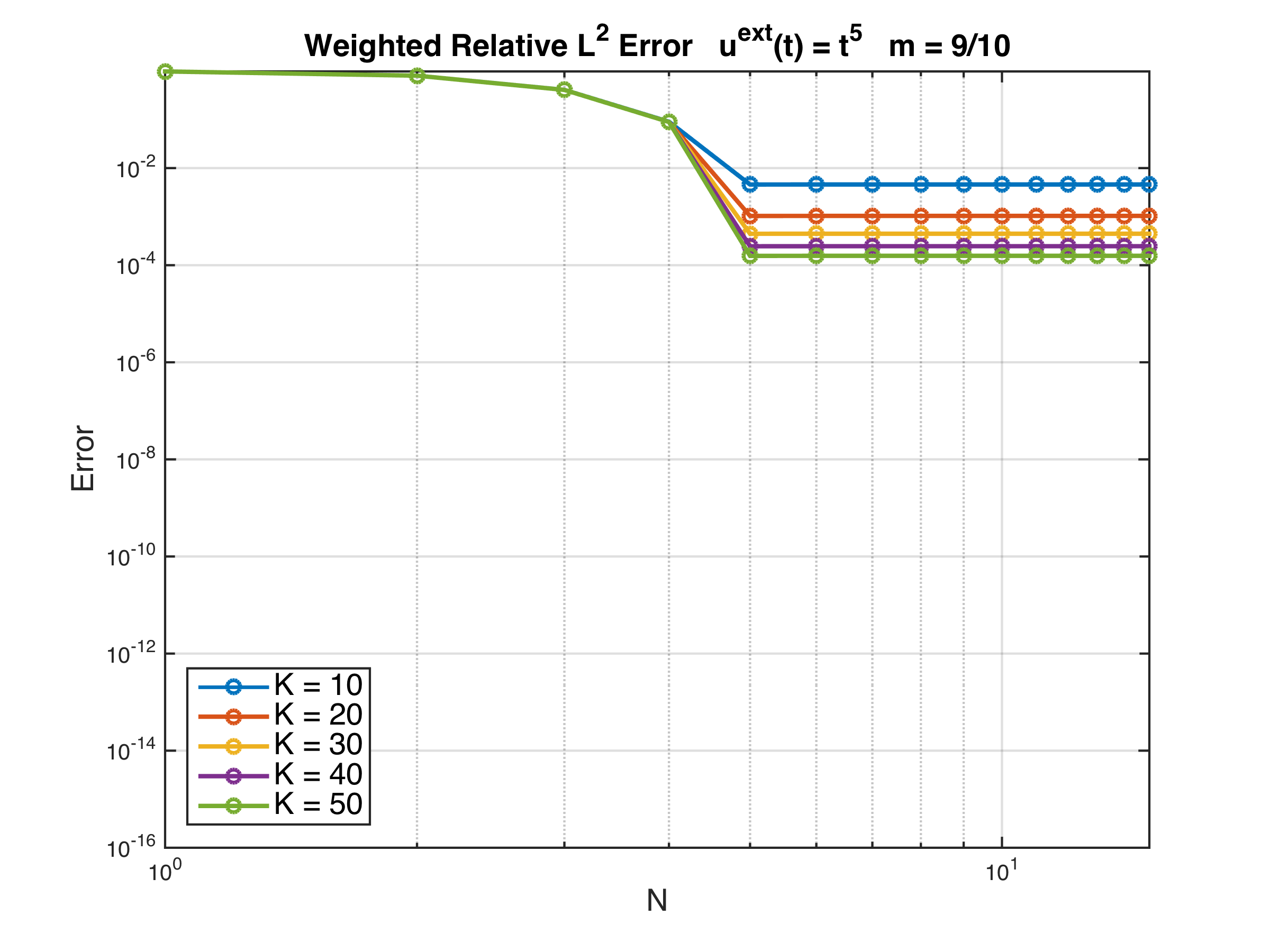}
\caption{Trapezoid rule}
\label{fig:left}
\end{subfigure}
\begin{subfigure}{0.49\textwidth}
\centering
\includegraphics[width = 6.5cm]{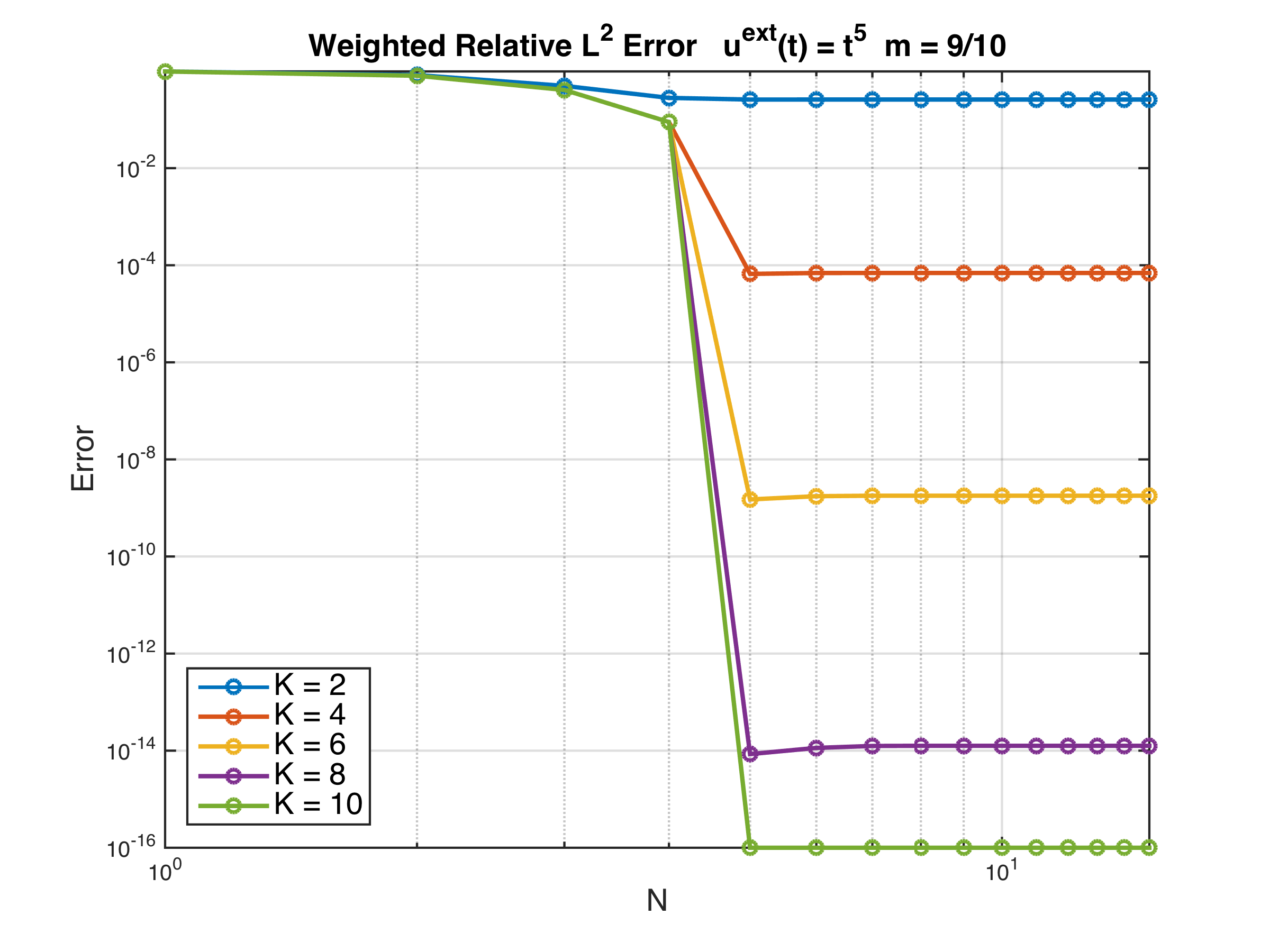}
\caption{Gauss-Legendre quadrature}
\label{fig:right}
\end{subfigure}
\caption{(a) Weighted relative $L^2$ error for the trapezoidal rule and our PG method applied to Example 6 with $m = 9/10$, where $K$ is the number of quadrature points used. \ \ (b) Weighted relative $L^2$ error for Gauss-Legendre quadrature and our PG method applied to Example 6 with $m = 9/10$.}
\label{fig:example6a}
\end{figure}

\FloatBarrier

We also apply our method to the same example using $m = 1/2$ and $m = 1/10$. The weighted relative 
$L^2$ error for both quadrature rules is plotted in Figures \ref{fig:example6b} and \ref{fig:example6c}. 
We see that the error plateaus in both the trapezoid and Gauss-Legendre cases, representing the level 
of error at which the quadrature rule dominates the approximation error of the PG method.

\FloatBarrier

\begin{figure}[ht!]
\centering
\begin{subfigure}{0.49\textwidth}
\centering
\includegraphics[width = 6.5cm]{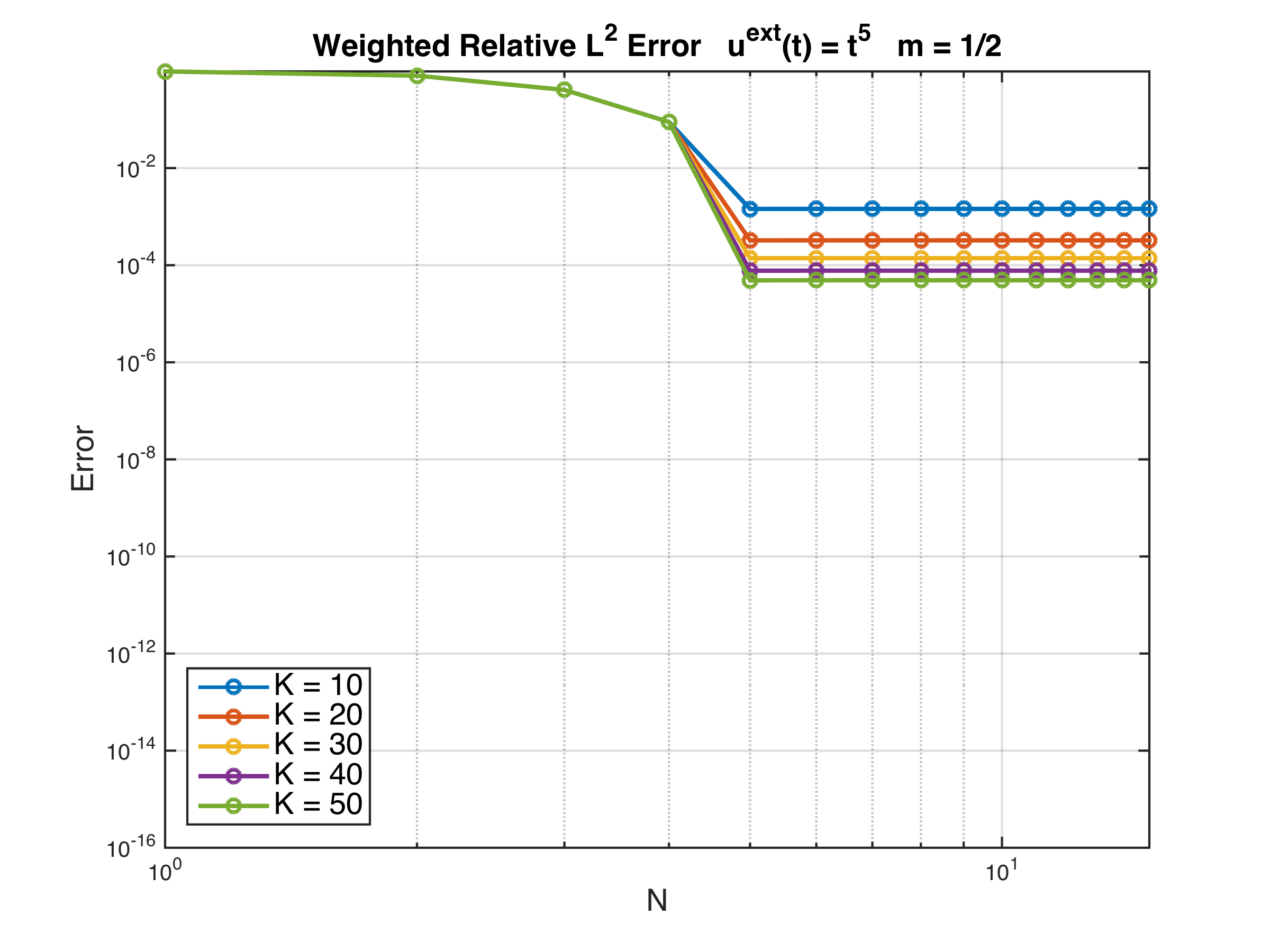}
\caption{Trapezoid rule}
\label{fig:left}
\end{subfigure}
\begin{subfigure}{0.49\textwidth}
\centering
\includegraphics[width = 6.5cm]{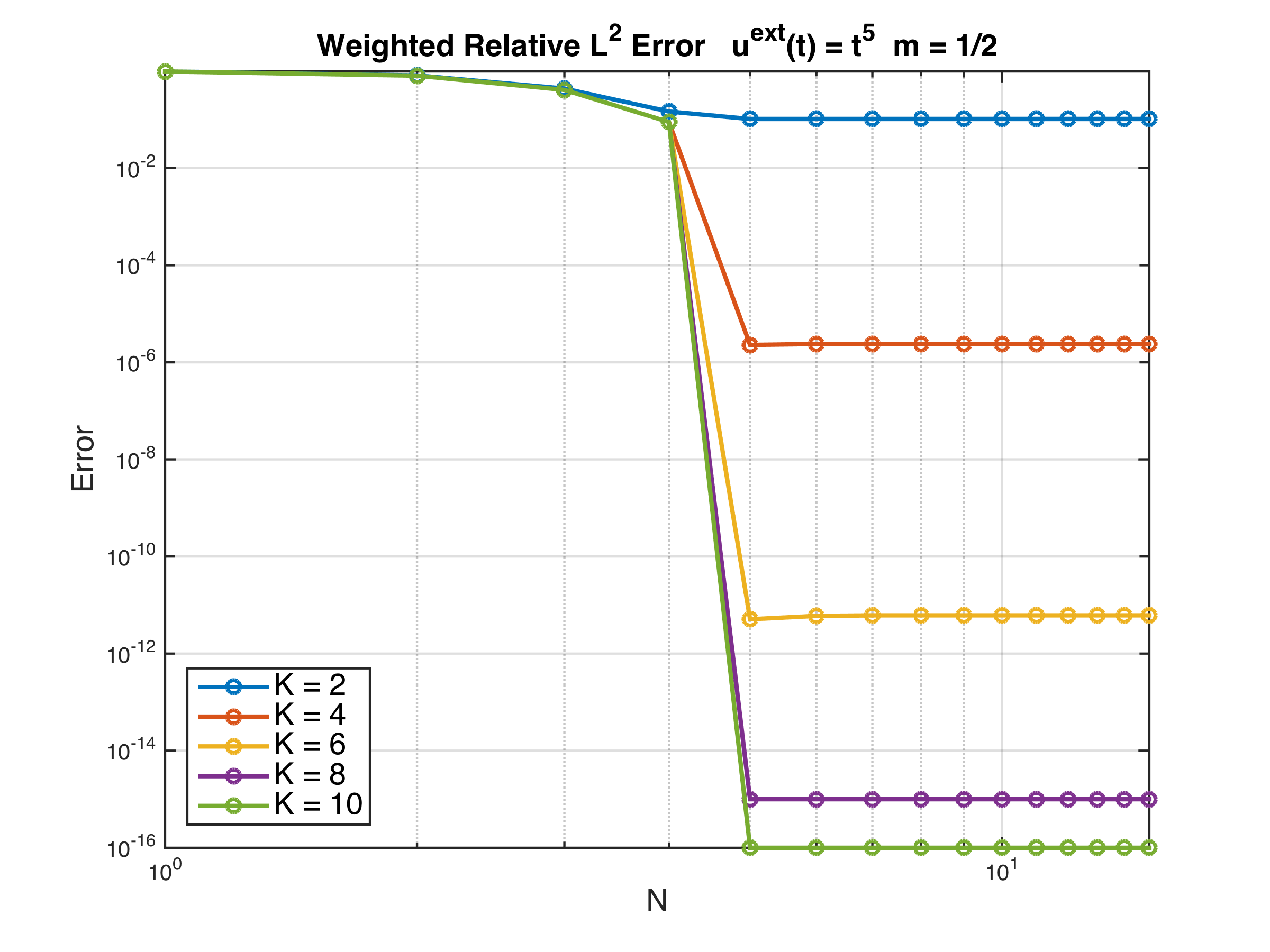}
\caption{Gauss-Legendre quadrature}
\label{fig:right}
\end{subfigure}
\caption{(a) Weighted relative $L^2$-error for the trapezoidal rule and our PG method applied to Example 5 with $m = 1/2$, where $K$ is the number of quadrature points used. \ \ (b) Weighted relative $L^2$-error for Gauss-Legendre quadrature and our PG method applied to Example 5 with $m = 1/2$.}
\label{fig:example6b}
\end{figure}

\begin{figure}[ht!]
\centering
\begin{subfigure}{0.49\textwidth}
\centering
\includegraphics[width = 6.5cm]{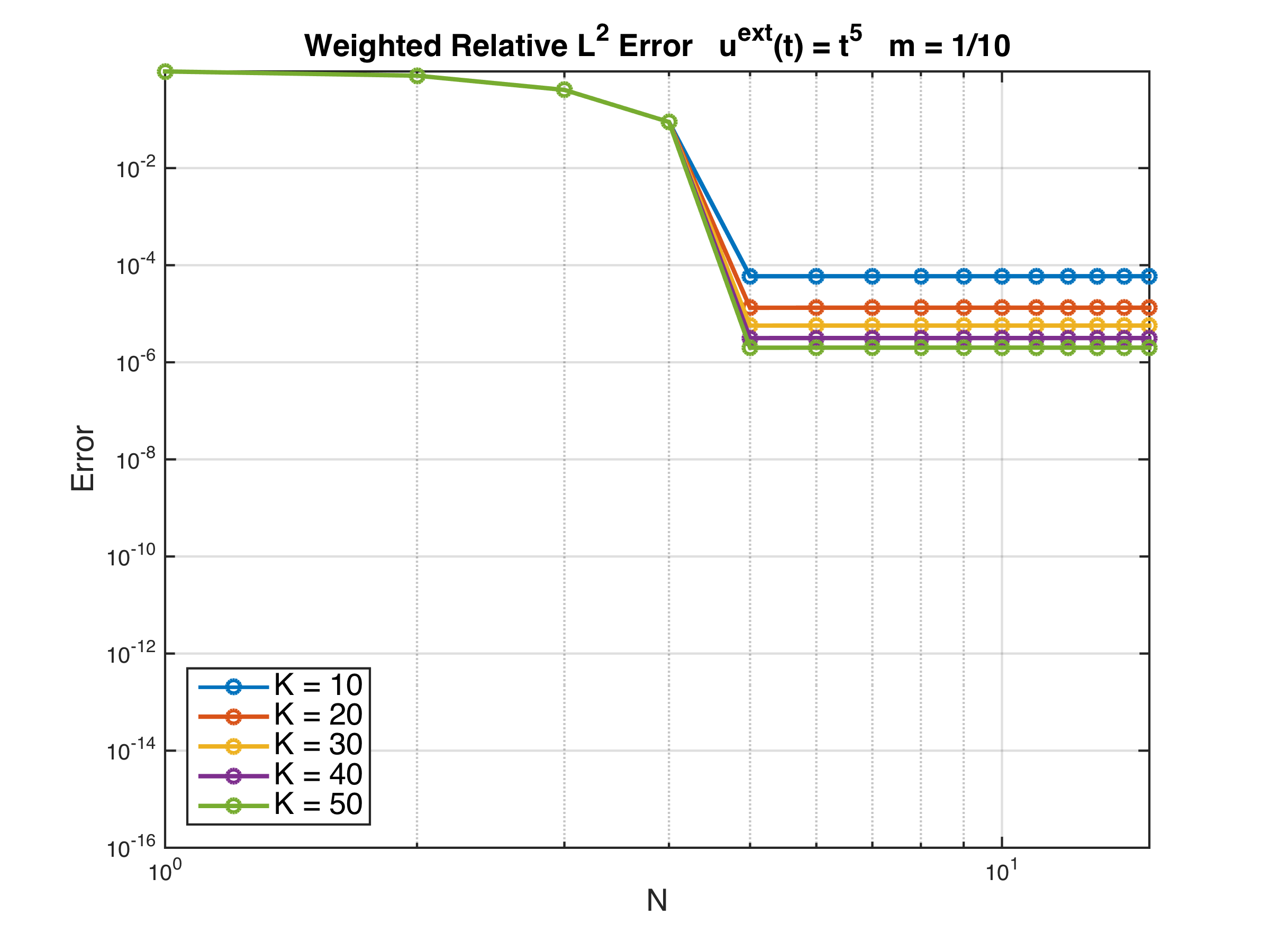}
\caption{Trapezoid rule}
\label{fig:left}
\end{subfigure}
\begin{subfigure}{0.49\textwidth}
\centering
\includegraphics[width = 6.5cm]{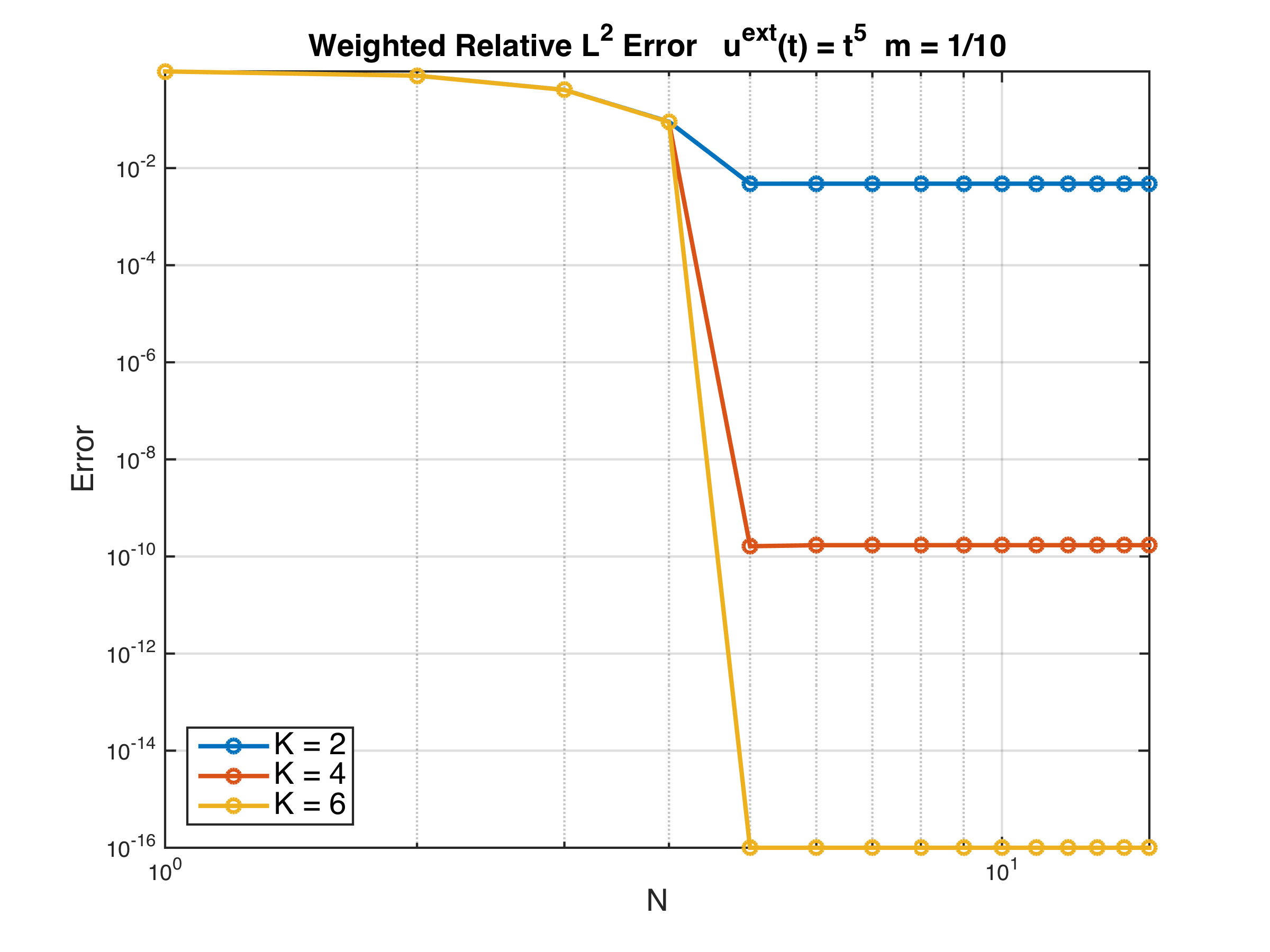}
\caption{Gauss-Legendre quadrature}
\label{fig:right}
\end{subfigure}
\caption{(a) Weighted relative $L^2$-error for the trapezoidal rule and our PG method applied to Example 6 with $m = 1/10$, where $K$ is the number of quadrature points used. \ \ (b) Weighted relative $L^2$-error for Gauss-Legendre quadrature and our PG method applied to Example 6 with $m = 1/10$.}
\label{fig:example6c}
\end{figure}

\FloatBarrier

\subsection{Example 7.}

We again solve equation \eqref{disteqn2} with the solution being a smooth function $u^{\text{ext}}(t) = t^3,$ 
with the distribution function $g(r) = \Gamma(4-r)\sinh(r).$ In this case, the right hand side function $f(t)$ is
\begin{align}
	\int_0^m g(r) {}_0\mcD_t^r u^{\text{ext}}(t) \ dr = \frac{t^{3-m}(t^m-\cosh(m)-\log(t) \sinh(m))}{(\log(t))^2-1} =: f(t).
\end{align}
We again choose the tuning parameter for the PG method to be $\alpha_1 = 1.$ We dispense with the 
Trapezoid rule and only use Gauss-Legendre quadrature in the remaining examples.

The weighted relative $L^2$ error for $m = 9/10$ and $1/10$ is plotted in Figure \ref{fig:example7}.

\begin{figure}[ht!]
\centering
\begin{subfigure}{0.49\textwidth}
\centering
\includegraphics[width = 6.5cm]{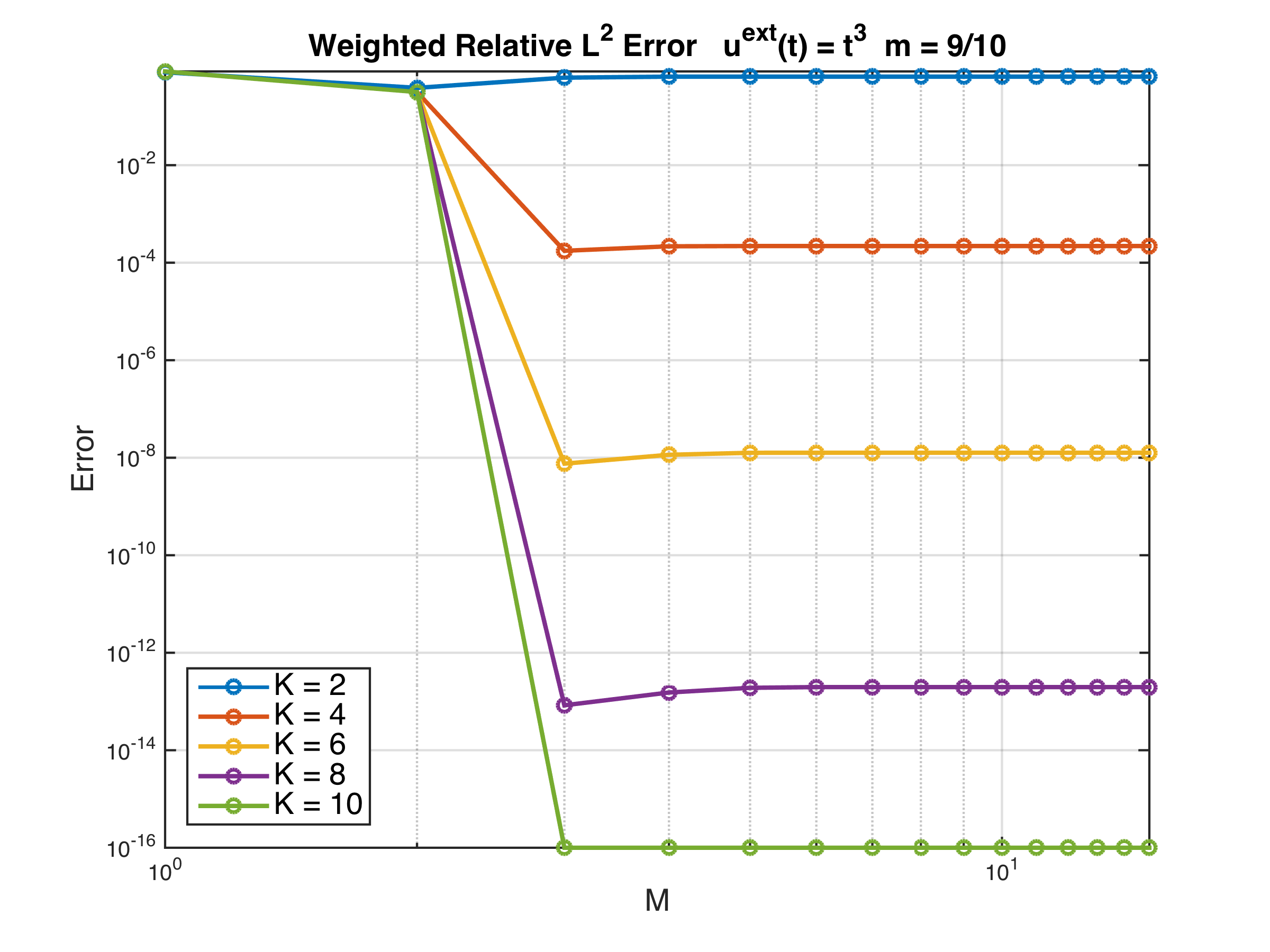}
\caption{$m = 9/10$}
\end{subfigure}
\begin{subfigure}{0.49\textwidth}
\centering
\includegraphics[width = 6.5cm]{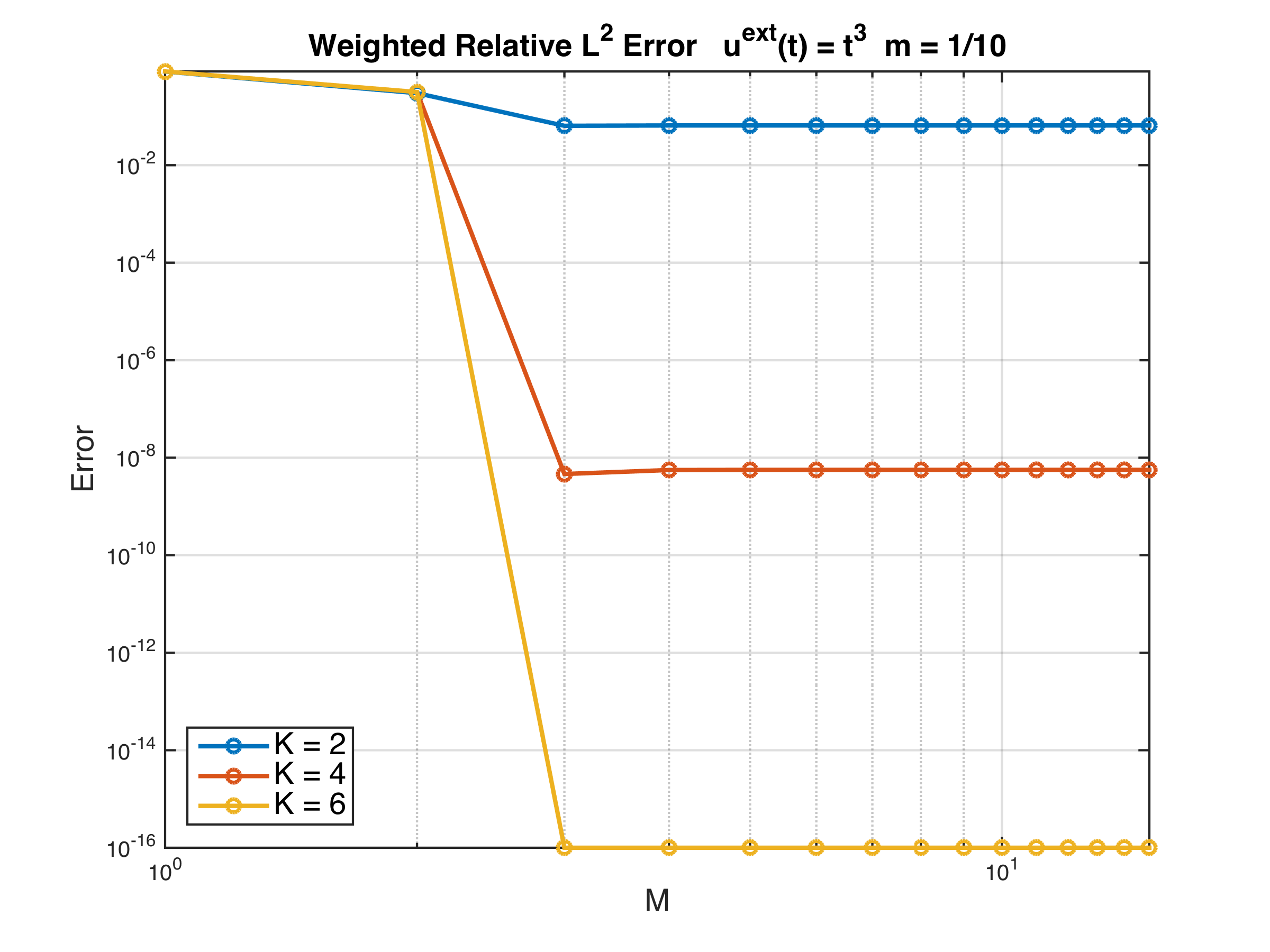}
\caption{$m = 1/10$}
\end{subfigure}
\caption{(a) Weighted relative $L^2$-error for Gauss-Legendre quadrature and our PG method applied to Example 7 with $m = 9/10$, where $K$ is the number of quadrature points used. \ \ (b) Weighted relative $L^2$-error for Example 7 with $m = 1/10.$}
\label{fig:example7}
\end{figure}

\FloatBarrier

\subsection{Example 8.}

Now we test a non-smooth example, where the fabricated solution is $u^{\text{ext}}(t) = t^{\lambda}$ 
with $\lambda = 2+1/3$ and the distribution function is $g(r) = \frac{\Gamma(1+ \lambda - r)}{\Gamma(\lambda+1)}.$ Then the right hand side function $f(t)$ is 
\begin{align}
	\int_0^m g(r) {}_0\mcD_t^r u^{\text{ext}}(t) \ dr = \frac{t^{\lambda - m}(t^m - 1)}{\log(t)} =: f(t).
\end{align}
We choose the tuning parameter for the PG method to be $\alpha_1 = 1/3.$

The weighted relative $L^2$ error for $m = 9/10, 1/2$ and $1/10$ is plotted in Figure \ref{fig:example8}.

\FloatBarrier

\begin{figure}[ht!]
\centering
\begin{subfigure}{0.49\textwidth}
\centering
\includegraphics[width = 6.5cm]{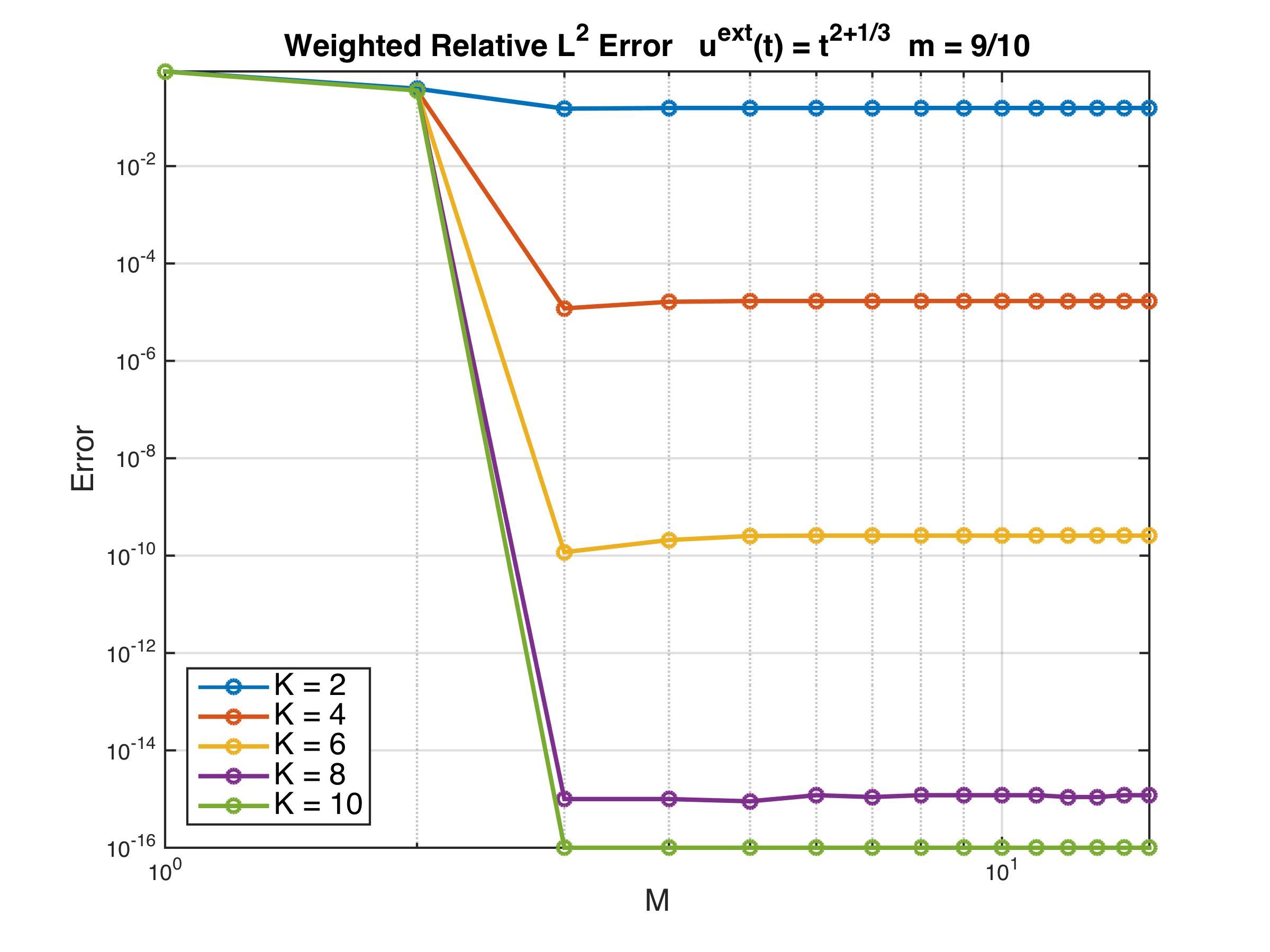}
\caption{$m = 9/10$}
\end{subfigure}
\begin{subfigure}{0.49\textwidth}
\centering
\includegraphics[width = 6.5cm]{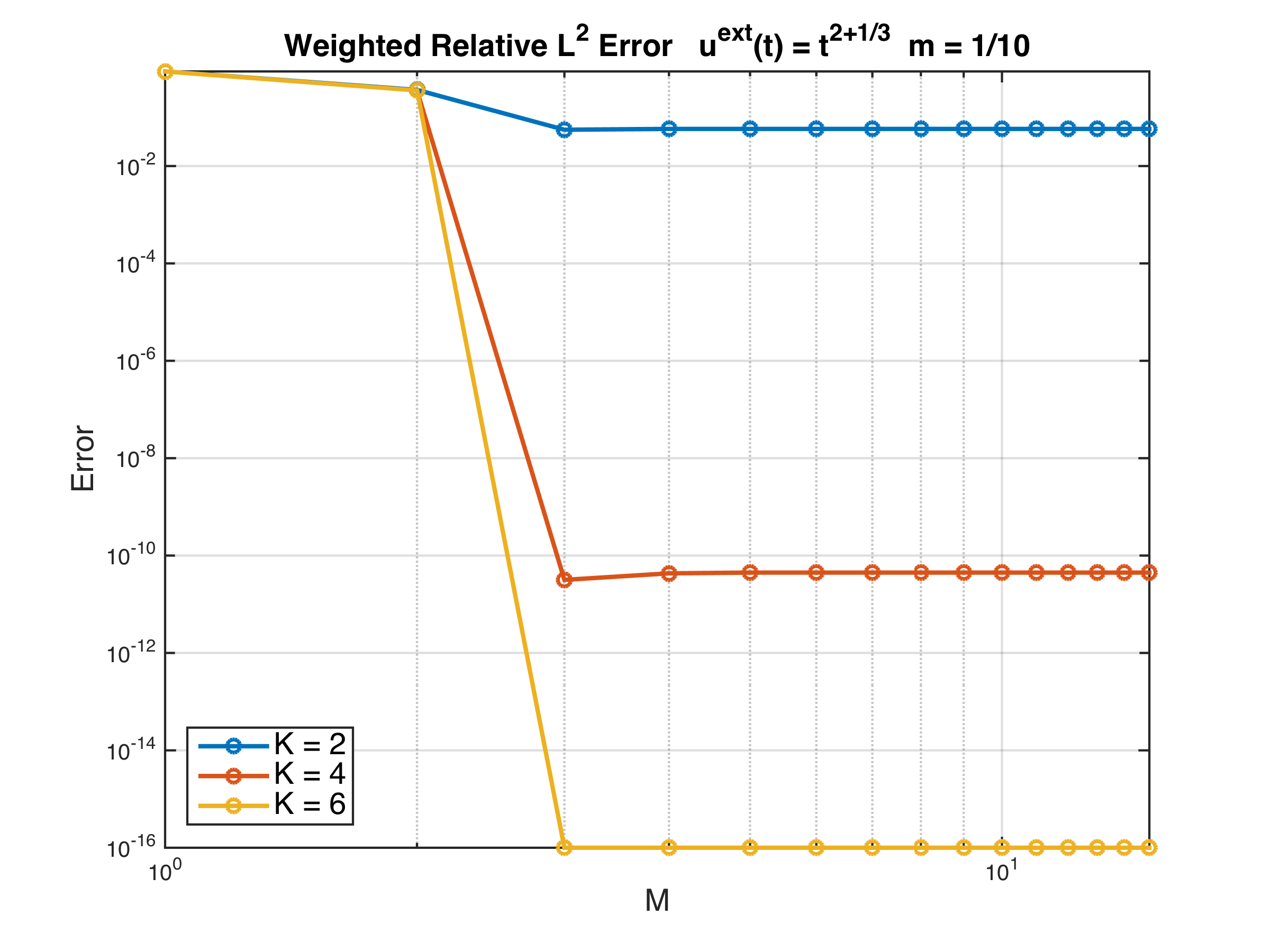}
\caption{$m = 1/10$}
\end{subfigure}
\caption{(a) Weighted relative $L^2$ error for Example 8 with $m = 9/10$. \ \ (b) Weighted relative $L^2$ error for Example 8 with $m = 1/10.$}
\label{fig:example8}
\end{figure}

\FloatBarrier


\section{Summary and Conclusion}

We have presented a new Laguerre Petrov-Galerkin spectral method for efficiently solving 
multi-term fractional initial value problems on the half line with order at most one. We demonstrated the 
tunable accuracy of the method using numerical experiments, and we showed that singularities of the type 
$t^\alpha$ are well-resolved using the GALF basis functions. We discussed the benefits resulting 
from the connection of the trial basis functions with the fractional Sturm-Liouville problems on the half line 
investigated  in \cite{FSL}. Our numerical results show that the method yields spectral convergence in the 
weighted $L^2$-norm on the half line, and that the convergence rate of the method is indeed 
sensitive to the tunable parameter $\alpha_1.$

We motivated the development of our highly efficient and well-conditioned PG method by solving the 
distributed order equation \eqref{disteqn2} following the idea of Diethelm and Ford, 
and we compared the results using both the trapezoid rule and Gauss-Legendre quadrature.

In the future, we will examine methods of analyzing our PG method and derive error estimates in the 
weighted relative $L^2$-norm to prove the spectral convergence of the method demonstrated in the 
numerical results sections. 


\bibliographystyle{unsrt}
\bibliography{mybibfile}

\begin{thebibliography}{10}

\bibitem{brunner}
H.~Brunner.
\newblock {\em Collocation methods for Volterra integral and related functional
  differential equations}, volume~15.
\newblock Cambridge University Press, Cambridge, UK, 11 2004.

\bibitem{delves}
L.M. Delves and J.L. Mohamed.
\newblock {\em Computational methods for integral equations}.
\newblock Cambridge University Press, Cambridge, UK, 12 1985.

\bibitem{podlubny}
I.~Podlubny.
\newblock {\em Fractional differential equations}.
\newblock Academic Press, Inc., San Diego, CA, 1999.

\bibitem{edwards}
J.T. Edwards, N.~J. Ford, and A.~C. Simpson.
\newblock The numerical solution of linear multi-term fractional differential
  equations: systems of equations.
\newblock {\em Journal of Computation and Applied Mathematics}, 148:401--418,
  2002.

\bibitem{celik}
C.~Celik and M.~Duman.
\newblock Crank-Nicolson method for the fractional diffusion equation with the
  Riesz fractional derivative.
\newblock {\em J. Comput. Phys.}, 231:1743--1750, 2012.

\bibitem{chen}
M.~Chen and W.~Deng.
\newblock Fourth order accurate scheme for the space fractional diffusion
  equations.
\newblock {\em SIAM J. Numer. Anal.}, 52:1418--1438, 2014.

\bibitem{ding}
H.~Ding, C.~Li, and Y.~Chen.
\newblock High-order algorithms for Riesz derivative and their applications.
\newblock {\em Abstr. Appl. Anal.}, 2013.

\bibitem{deng}
W.~Deng.
\newblock Finite element method for the space and time fractional Fokker-Planck
  equation.
\newblock {\em SIAM J. Numer. Anal.}, 47:204--226, 2008/09.

\bibitem{ervin}
V.J. Ervin and J.P. Roop.
\newblock Variational solution of fractional advection dispersion equations on
  bounded domains in $\mathbb{R}^d$.
\newblock {\em Numer. Methods Partial Differential Equations}, 23:256--281,
  2007.

\bibitem{zayer1}
M.~Zayernouri, M.~Ainsworth, and G.~E. Karniadakis.
\newblock A unified Petrov-Galerkin spectral method for fractional PDEs.
\newblock {\em Computer Methods in Applied Mechanics and Engineering},
  283:1545--1569, 2015.

\bibitem{zayer2}
M.~Zayernouri and G.~E. Karniadakis.
\newblock Exponentially accurate spectral and spectral element methods for
  fractional ODEs.
\newblock {\em Journal of Computational Physics}, 257:460--480, 2014.

\bibitem{zayernouri2015tempered}
M.~Zayernouri, M.~Ainsworth, and G.~E. Karniadakis.
\newblock Tempered fractional Sturm--Liouville eigen-problems.
\newblock {\em SIAM Journal on Scientific Computing}, 37(4):A1777--A1800, 2015.

\bibitem{Zayernouri14-SIAM-Collocation}
M.~Zayernouri and G.~E. Karniadakis.
\newblock Fractional spectral collocation method.
\newblock {\em SIAM Journal on Scientific Computing}, 36(1):A40--A62, 2014.

\bibitem{Zayernouri14-SIAM-Frac-Advection}
M.~Zayernouri and G.~E. Karniadakis.
\newblock Discontinuous spectral element methods for time-and space-fractional
  advection equations.
\newblock {\em SIAM Journal on Scientific Computing}, 36(4):B684--B707, 2014.

\bibitem{FSL}
H.~Khosravian-Arab, Mehdi Dehghan, and M.R. Eslahchi.
\newblock Fractional Sturm-Liouville boundary value problems in unbounded
  domains: Theory and applications.
\newblock {\em Journal of Computational Physics}, 229:526--560, 2015.

\bibitem{zhang}
Zhongqiang Zhang, Fanhai Zeng, and George~Em Karniadakis.
\newblock Optimal error estimates of spectral Petrov-Galerkin and collocation
  methods for initial value problems of fractional differential equations.
\newblock {\em SIAM J. Numer. Anal.}, 53:2074--2096, 2015.

\bibitem{baleanu}
D.~Baleanu, A.H. Bhrawy, and T.M. Taha.
\newblock A modified generalized Laguerre spectral method for fractional
  differential equations on the half line.
\newblock {\em Abstract and Applied Analysis}, 2013, 2013.

\bibitem{bhrawy}
A.H. Bhrawy, D.~Baleanu, and L.M. Assas.
\newblock Efficient generalized Laguerre spectral methods for solving
  multi-term fractional differential equations on the half line.
\newblock {\em Journal of Vibration and Control}, 2013.

\bibitem{diethelm}
Kai Diethelm and Neville~J. Ford.
\newblock Numerical analysis for distributed-order differential equations.
\newblock {\em J. of Comp. and App. Math.}, 225, 2009.

\bibitem{oscillator}
T.~M. Atanackovic, M.~Budincevic, and S.~Pilipovic.
\newblock On a fractional distributed-order oscillator.
\newblock {\em Journal of Physics A: Mathematical and General}, 38:6703--6713,
  2005.

\bibitem{ear}
M.~Naghibolhosseini.
\newblock {\em Estimation of outer-middle ear transmission using
  \uppercase{DPOAE}s and fractional-order modeling of human middle ear}.
\newblock PhD thesis, City University of New York, NY., 2015.

\bibitem{caputo1}
M.~Caputo.
\newblock Distributed order differential equations modelling dielectric
  induction and diffusion.
\newblock {\em Fract. Calc. Appl. Anal.}, 4:421--442, 2001.

\bibitem{caputo2}
M.~Caputo.
\newblock Diffusion with space memory modelled with distributed order space
  fractional differential equations.
\newblock {\em Ann. Geophys.}, 46:223--234, 2003.

\bibitem{sokolov}
I.M. Sokolov, A.V. Chechkin, and J.~Klafter.
\newblock Distributed order fractional kinetics.
\newblock {\em Acta Phys. Pol. B}, 35:1323--1341, 2004.

\bibitem{samko}
S.G. Samko, A.A. Kilbas, and O.I. Marichev.
\newblock {\em Fractional integrals and derivatives}.
\newblock Gordon and Breach Science Publishers, Yverdon, 1993.

\bibitem{Li2009SpaceTime}
X.~Li and C.~Xu.
\newblock A space-time spectral method for the time fractional diffusion
  equation.
\newblock {\em SIAM Journal on Numerical Analysis}, 47(3):2108--2131, 2009.

\bibitem{toeplitz}
P.G. Martinsson, V.~Rokhlin, and M.~Tygert.
\newblock A fast algorithm for the inversion of general Toeplitz matrics.
\newblock {\em Computers and Mathematics with Applications}, 50:741--752, 2005.

\bibitem{shen}
J.~Shen, T.~Tang, and L.-L. Wang.
\newblock {\em Spectral Methods: Algorithms, Analysis and Applications}.
\newblock Springer, 8 2011.

\bibitem{hesthaven}
J.~Hesthaven, S.~Gottlieb, and D.~Gottlieb.
\newblock {\em Spectral Methods for Time-Dependent Problems}.
\newblock Cambridge University Press, Cambridge, UK, 2007.

\bibitem{kharazmi}
E.~Kharazmi, M.~Zayernouri, and G.~E. Karniadakis.
\newblock Petrov-Galerkin and spectral collocation methods for distributed
  order differential equations.
\newblock {\em arXiv preprint arXiv:1604.08650}, 2016.

\end{thebibliography}

\end{document}